\documentclass[a4paper, 12pt]{article}
\usepackage{amsfonts}
\usepackage{fancyhdr}
\usepackage{graphicx}
\usepackage{epsfig}
\usepackage{amssymb}
\usepackage{amsthm}
 \usepackage{mathrsfs,amsfonts,amsmath,mathrsfs}
 \usepackage{color}

\makeatletter

\@addtoreset{equation}{section} \makeatother
\newtheorem{Theorem}{Theorem}[section]

\newtheorem{Lemma}[Theorem]{Lemma}

\newtheorem{Assumption}[Theorem]{Assumption}

\begin{document}
\title{\textbf{Convergence and exponential stability of modified truncated Milstein method for stochastic differential equations
\footnote{Supported by Natural Science Foundation Beijing Municipality(1192013).}}}
\author{ Yu Jiang\quad and\quad Guangqiang Lan\footnote{Corresponding author: Email:
langq@buct.edu.cn.}
\\ \small College of Mathematics and Physics, Beijing University of Chemical Technology, Beijing 100029, China}

\date{}

\maketitle

\begin{abstract}
In this paper, we develop a new explicit scheme called modified truncated Milstein method which is motivated by truncated Milstein method proposed by Guo (2018) and modified truncated Euler-Maruyama method introduced by Lan (2018). We obtain the strong convergence of the scheme under local boundedness and Khasminskii-type conditions, which are relatively weaker than the existing results, and we prove that the convergence rate could be arbitrarily close to 1 under given conditions. Moreover, exponential stability of the scheme is also considered while it is impossible for truncated Milstein method introduced in Guo(2018). Three numerical experiments are offered to support our conclusions.
\end{abstract}



\section{Introduction}
\quad Most equations can not be solved explicitly, so numerical scheme of stochastic differential equations(SDEs) are very significant. There has been a great deal of literature on numerical methods of the corresponding SDEs. For instance, Euler-Maruyama (EM) type methods have been widely considered in the past decades. In general, the convergence rates of EM type methods are expected to be at most $\frac{1}{2}$. Indeed, the classical EM method converges to the corresponding exact solution with rate $\frac{1}{2}$ under global Lipschitz condition. However, global Lipschitz condition is so restrictive that many interesting cases are not included in this setting. So people pay more and more attention to local Lipschitz condition. Higham et al. \cite{H} considered the strong convergence for numerical approximations under local Lipschitz condition (plus some other suitable conditions) for the first time. Since then, many convergence results of Euler type are obtained under local Lipschitz condition. For example, Mao(2015) in \cite{AB} developed a new explicit numerical simulation method, called truncated EM method, and investigated the strong convergence of it. Then in \cite{M1}, the author obtained convergence rates of truncated EM method under similar conditions. Motivated by two papers of Mao, Lan and Xia introduced in \cite{AC} the modified truncated Euler-Maruyama (MTEM) method and obtained the strong convergence rate under relatively weaker conditions. In spite of its explicity, the convergence rate of the EM-type method can only close to $\frac{1}{2}$ from below (but can not reach).

In order to obtain higher convergence rate, a scheme attributed to Milstein came into being (see \cite{AF}). \cite{AH} proposed the tamed Milstein method which allows the drift coefficients to be non-globally Lipschitz continuous but the diffusion coefficients are still globally Lipschitz continuous. \cite{AI} introduced the truncated Milstein method of stochastic differential equations with both of drift and diffusion coefficients satisfying polynomial growth condition. In that paper, the authors proved the strong convergence rate is close to 1. Motivated by them, we will loose their conditions and develop the modified truncated Milstein method, which only need the first and second order derivative of the drift and diffusion coefficients to be locally bounded (see the following Assumption 2.1). What's more, we do not use the following weak monotonicity condition
\begin{align*}
\langle x-y,f(x)-f(y)\rangle+\frac{p-1}{2}|g(x)-g(y)|^{2}\le K|x-y|^{2},
\end{align*}
which is critical in \cite{AI}. Moreover, we allow that the derivative of drift and diffusion coefficients do not satisfy the polynomial growth condition.

Besides, to the best of our knowledge, there are no stability results for the Milstein type numerical methods so far. We give the exponential stability of our scheme which is impossible for the truncated Milstein method in \cite{AI}.

The organization of the paper is as the following. In Section 2, we will introduce some notations and assumptions at first, then the modified truncated Milstein method will be presented. In Section 3, several useful lemmas will be presented for the proof of strong convergence. We mainly discuss the convergence of our scheme and obtain convergence rate in Section 4. In Section 5, we prove exponential stability of our scheme. Finally, we present some examples to verify our conclusions in Section 6.

\section{Modified truncated Milstein method}
\quad Throughout this paper, unless otherwise specified, let $(\Omega,\mathscr{F},\mathbb{P})$ be a complete probability space with a filtration $\{\mathscr{F}_{t}\}_{t\ge 0}$ satisfying the usual conditions (that is, it is right continuous and increasing while $\mathscr{F}_{0}$ contains all $\mathbb{P}$-null sets). Let $B(t)$ be an m-dimensional Brownian motion. $A^{T}$ denotes the transpose of a vector or matrix $A$. $|x|$ is the Euclidean norm of $x\in \mathbb{R}^{d}$. If $A$ is a matrix, $|A|=\sqrt{trace(A^{T}A)}$ denotes its trace norm.

We consider a d-dimensional SDE
\begin{equation}\label{B1}
dx(t)=f(x(t))dt+\sum_{j=1}^{m}g_{j}(x(t))dB^{j}(t)
\end{equation}
on $t\in[0,T]$ with the initial value $x(0)=x_{0}\in \mathbb{R}^{d}$, where
\begin{align*}
f:\mathbb{R}^{d}\rightarrow \mathbb{R}^{d}\quad \textrm{and}\quad g:\mathbb{R}^{d}\rightarrow \mathbb{R}^{d\times m}
\end{align*}
are measurable functions.

For $j_{1},j_{2}=1,2,...,m$, define
\begin{align*}
L^{j_{1}}g_{j_{2}}(x)=\sum_{l=1}^{d}g_{l,j_{1}}(x)\frac{\partial g_{j_{2}}(x)}{\partial x^{l}}.
\end{align*}
It is called commutative if $L^{j_{1}}g_{j_{2}}(x)=L^{j_{2}}g_{j_{1}}(x)$.

And for $l=1,2,...,d$, set
\begin{align*}
 \ f'_{l}(x)=(\frac{\partial f_{l}(x)}{\partial x^{1}},\frac{\partial f_{l}(x)}{\partial x^{2}},...,\frac{\partial            f_{l}(x)}{\partial x^{d}})\quad \textrm{and} \quad f''_{l}(x)=(\frac{\partial^{2}f_{l}(x)}{\partial x^{j}\partial x^{i}})_{i,j},i,j=1,2,...,d.
    \end{align*}
For $l=1,2,...,d, n=1,2,...,m$, $g'_{l,n}(x)$ and $g''_{l,n}(x)$ can be defined in the same way.

 To investigate the convergence of given numerical method, we impose the following standing hypotheses in this paper.

\begin{Assumption}\label{A3} (\textbf{local boundedness condition}) For any $R>0$, there are $K_{R}>0$ and $C'>0$ (independent of $R$) such that
 \begin{equation}
  \label{D}|f'_{l}(x)|\vee|f''_{l}(x)|\vee|g'_{l,n}(x)|\vee|g''_{l,n}(x)|\le K_{R}
   \end{equation}
    holds for all $ x\in \mathbb{R}^{d} $ with $ |x|\le R $.
     \end{Assumption}

\begin{Assumption}\label{A2} (\textbf{Khasminskii-type condition}) There are positive $K$ and $p>2$ such that
 \begin{equation}\label{B17}
  \langle x,f(x)\rangle + \frac{p-1}{2}\sum_{j=1}^{m}|g_{j}(x)|^{2}\le K(1+|x|^{2})
   \end{equation}
    \ for all $x\in \mathbb{R}^{d}$.
     \end{Assumption}

For $j=1,...,m$ and $l=1,...,d$, define the derivative of the vector $g_{j}(x)$ with respect to $x^{l}$ by
\begin{align*}
G_{j}^{l}(x):=\frac{\partial}{\partial x^{l}}g_{j}(x)=(\frac{\partial g_{1,j}(x)}{\partial x^{l}},\frac{\partial g_{2,j}(x)}{\partial x^{l}},...,\frac{\partial g_{d,j}(x)}{\partial x^{l}})^{T}.
\end{align*}
To define the modified truncated Milstein method, we choose a number $\Delta^{\ast}\in (0,1]$ and a strictly decreasing function $h:(0,\Delta^{\ast}]\rightarrow(0,+\infty)$ such that
\begin{equation}\label{B8}
\lim_{\Delta\rightarrow 0}h(\Delta)=\infty.
\end{equation}
For a given step size $\Delta\in(0,1)$ and any $x\in\mathbb{R}^{d}$, define the modified truncated functions by
\begin{equation}\label{mtf1}
 \tilde{f}(x)=\begin{cases}
 f(x) & |x|\le h(\Delta),\\
   \frac{|x|}{h(\Delta)}f(\frac{x}{|x|}h(\Delta)) & |x|>h(\Delta).
    \end{cases}
     \end{equation}
\begin{equation}\label{mtf2}
 \tilde{g}(x)=\begin{cases}
 g(x) & |x|\le h(\Delta),\\
   \frac{|x|}{h(\Delta)}g(\frac{x}{|x|}h(\Delta)) & |x|>h(\Delta).
    \end{cases}
     \end{equation}
\begin{equation}\label{mtf3}
 \tilde{G}_{j}^{l}(x)=G_{j}^{l}\Bigl((|x|\land h(\Delta))\frac{x}{|x|} \Bigl),j=1,2,...,m,l=1,2,...,d
  \end{equation}
where we set $\frac{x}{|x|}=0$ if $x=0$.

If $g$ is commutative, then similar to \cite{AI}, the modified truncated Milstein method can be defined by $Y_0=x(0)=x_0,$
\begin{align}
Y_{k+1} & =Y_{k}+\tilde{f}(Y_{k})\Delta+\sum_{j=1}^{m}\tilde{g}_{j}(Y_{k})\Delta B_{k}^{j}\notag+\frac{1}{2}\sum_{j_{1}=1}^{m}\sum_{j_{2}=1}^{m}\sum_{l=1}^{d}\tilde{g}_{l,j_{1}}(Y_{k})\tilde{G}_{j_{2}}^{l}(Y_{k})\Delta B_{k}^{j_{2}}\Delta B_{k}^{j_{1}}\\
&\quad-\frac{1}{2}\sum_{j=1}^{m}\sum_{l=1}^{d}\tilde{g}_{l,j}(Y_{k})\tilde{G}_{j}^{l}(Y_{k})\Delta, k\ge0.
\end{align}

Notice that if we replace $\tilde{f},\tilde{g}$ with $f$ and $g$, respectively, then it goes back to classical Milstein method.

The two continuous-time versions of the modified truncated Milstein method are defined as the following:
\begin{equation}
\bar{Y}(t)=\sum_{k=0}^{\infty}Y_{k}1_{[k\Delta,(k+1)\Delta)}(t),t\ge 0,
\end{equation}
and
\begin{equation}\label{B4}
\aligned Y(t)&=x_{0}+\int_{0}^{t}\tilde{f}(\bar{Y}(s))ds+\sum_{j=1}^{m}\int_{0}^{t}\tilde{g}_{j}(\bar{Y}(s))dB^{j}(s)\\&
\quad+\sum_{j_{1}=1}^{m}\int_{0}^{t}\sum_{j_{2}=1}^{m}L^{j_{1}}\tilde{g}_{j_{2}}(\bar{Y}(s))\Delta B^{j_{2}}(s)dB^{j_{1}}(s)\endaligned
\end{equation}
where $\Delta B^{j}(s)=\sum_{k=0}^{\infty}1_{[k\Delta,(k+1)\Delta)}(s)(B^{j}(s)-B^{j}(k\Delta))$. It is obvious that all of the above three modified truncated Milstein methods equals at the grid points, that is, $Y_k=\bar{Y}(k\Delta)=Y(k\Delta), \forall k\ge0$.

\section{Some useful lemmas}

\noindent \quad In this section, we state some useful lemmas which will be used in the next section.

Firstly, we state a known result (see e.g. \cite{AA}) as a lemma for the use of this paper.
\begin{Lemma}
Under $(\ref{B17})$ and $(\ref{BA31})$, the SDE $(\ref{B1})$ has a unique global solution $x(t)$. Moreover,
\begin{equation}\label{B16}
\sup_{0\le t\le T}\mathbb{E}|x(t)|^{p}<\infty,\quad \forall T>0.
\end{equation}
\end{Lemma}

Usually, local Lipschitz condition of the coefficients $f$ and $g$ needs to be satisfied. Indeed, we have the following
\begin{Lemma}\label{l0}
If Assumption \ref{A3} holds, then
\begin{equation}\label{BA31}
|f(x)-f(y)|\vee|g_{j}(x)-g_{j}(y)|\le dK_{R}|x-y|,\
|L^{j_{1}}g_{j_{2}}(x)-L^{j_{1}}g_{j_{2}}(y)|\le \bar{K}_{R}|x-y|
\end{equation}
for all $x,y\in \mathbb{R}^{d}$ with $ |x|\vee|y|\le R $ and $j,j_{1},j_{2}=1,2,...,m$, where $\bar{K}_{R}=2(dK_{R}(RK_{R}+|g(0)|\vee1)\vee dK_{R}^{2})$.
\end{Lemma}
\begin{proof}
By mean value theorem of vector valued function, we have
\begin{align*}
|f(x)-f(y)&|\le |F'(x+\theta(y-x))||x-y|\\&
=\sqrt{\sum_{l=1}^{d}\Bigl|f'_{l}(((x+\theta(y-x))\Bigl|^{2}}|x-y|\\&\le \sqrt{d}K_{R}|x-y|,
\end{align*}
where $\theta\in(0,1)$ and $F'(x)$ is Jacobian Matrix of $f(x)$. Similarly, we can get local Lipschitz condition of $g(x)$ with local Lipschitz coefficient $dK_R$.

Now
\begin{align*}
|L^{j_{1}}g_{j_{2}}(x)-L^{j_{1}}g_{j_{2}}(y)|&\le\Bigl|\sum_{l=1}^{d}g_{l,j_{1}}(x)\frac{\partial g_{j_{2}}(x)}{\partial x^{l}}-\sum_{l=1}^{d}g_{l,j_{1}}(x)\frac{\partial g_{j_{2}}(y)}{\partial y^{l}}\Bigl|\\&\quad+\Bigl|\sum_{l=1}^{d}g_{l,j_{1}}(x)\frac{\partial g_{j_{2}}(y)}{\partial y^{l}}-\sum_{l=1}^{d}g_{l,j_{1}}(y)\frac{\partial g_{j_{2}}(y)}{\partial y^{l}}\Bigl|\\
&\le\sum_{l=1}^{d}|g_{l,j_{1}}(x)|\cdot|\frac{\partial g_{j_{2}}(x)}{\partial x^{l}}-\frac{\partial g_{j_{2}}(y)}{\partial y^{l}}|\\&\quad+\sum_{l=1}^{d}\Bigl|g_{l,j_{1}}(x)-g_{l,j_{1}}(y)\Bigl|\cdot\Bigl|\frac{\partial g_{j_{2}}(y)}{\partial y^{l}}\Bigl|\\
&\le (RK_{R}+|g(0)|)\sum_{l=1}^{d}\sqrt{\sum_{i=1}^{d}\sum_{k=1}^{d}\Bigl|\frac{\partial g_{i,j_{2}}(\xi)}{\partial x^{l}\partial x^{k}}\Bigl|_{\xi=x+\theta(y-x)}^{2}}\cdot|x-y|\\&\quad
+K_{R}\sum_{l=1}^{d}\sqrt{\sum_{i=1}^{d}\Bigl|\frac{\partial g_{l,j_{1}}(\tilde{\xi})}{\partial x^{i}}\Bigl|_{\tilde{\xi}=x+\theta'(y-x)}^{2}}\cdot|x-y|\\
&\le (RK_{R}+|g(0)|)\sqrt{d\sum_{l=1}^{d}\sum_{i=1}^{d}\sum_{k=1}^{d}\Bigl|\frac{\partial g_{i,j_{2}}(\xi)}{\partial x^{l}\partial x^{k}}\Bigl|_{\xi=x+\theta(y-x)}^{2}}\cdot|x-y|\\&\quad
+K_{R}\sqrt{d\sum_{l=1}^{d}\sum_{i=1}^{d}\Bigl|\frac{\partial g_{l,j_{1}}(\tilde{\xi})}{\partial x^{i}}\Bigl|_{\tilde{\xi}=x+\theta'(y-x)}^{2}}\cdot|x-y|\\
&=(RK_{R}+|g(0)|)\sqrt{d\sum_{i=1}^{d}|g''_{i,j_{2}}(\xi)|^{2}}\cdot|x-y|\\&\quad+K_{R}\sqrt{d\sum_{i=1}^{d}|g'_{i,j_{2}}(\xi)|^{2}}\cdot|x-y|\\
&\le dK_{R}(RK_{R}+|g(0)|)|x-y|+dK_{R}^{2}|x-y|\\& \le 2(dK_{R}(RK_{R}+|g(0)|)\vee dK_{R}^{2})|x-y|.
\end{align*}

We complete the proof.
\end{proof}

For the modified truncated function, we have the following
\begin{Lemma}\label{A4} Suppose the local Lipschitz condition $(\ref{BA31})$ holds.
Then for any fixed $\Delta>0$ (sufficiently small), we have
\begin{equation}\label{B6}
|\tilde{f}(x)-\tilde{f}(y)|\vee|\tilde{g}_{j}(x)-\tilde{g}_{j}(y)|\vee|L^{j_{1}}\tilde{g}_{j_{2}}(x)-L^{j_{1}}\tilde{g}_{j_{2}}(y)|\le 4\bar{K}_{h(\Delta)}|x-y|
  \end{equation}
where
\begin{equation}\label{B13}
L^{j_{1}}\tilde{g}_{j_{2}}(x):=\sum_{l=1}^{d}\tilde{g}_{l,j_{1}}(x)\tilde{G}^{l}_{j_{2}}(x).
\end{equation}
\end{Lemma}

The proof is similar to Lemma 2.2 of \cite{AC}, so we omit it here.

From now on, we choose $h$ such that
\begin{equation}\Delta\bar{K}^{\frac{9}{4}}_{h(\Delta)}\le 1,\forall \Delta\in (0,\Delta^*).\end{equation}
Note that such $h$ always exists (see \cite{AC}, Remark 2.1).
\begin{Lemma}\label{A5}
Assume that $\ref{A2}$ holds. Then for all $\Delta\in(0,\Delta^{\ast}]$ and any $x\in\mathbb{R}^{d}$,
\begin{equation}\label{B3}
\Bigl\langle x,\tilde{f}(x)\Bigl\rangle + \frac{p-1}{2}\sum_{j=1}^{m}|\tilde{g}_{j}(x)|^{2}\le 2K(1+|x|^{2}).
\end{equation}
\end{Lemma}
For the proof, see Lemma 3.1 in \cite{AC}.

\begin{Lemma}\label{A6}
Let $(\ref{B17})$, $(\ref{BA31})$ and $(\ref{B8})$ hold. Then for any $p\ge2$, $T>0$ and $\Delta^{\ast}$ sufficiently small, there exists $C>0$ such that
\begin{align*}
\sup_{0<\Delta\le\Delta^{\ast}}\sup_{0\le t\le T}\mathbb{E}|Y(t)|^{p}\le C.
\end{align*}
\begin{proof}
Fix the step size $\Delta\in(0,\Delta^{\ast}]$ arbitrarily. For any $t\ge 0$, there exists a unique integer $k\ge 0$ such that $t_{k}\le t\le t_{k+1}$, where $t_{k}=k\Delta$. By the elementary inequality, we derive from (\ref{B4}) that
\begin{align*}
\mathbb{E}|Y(t)-\bar{Y}(t)|^{p}& \le C \mathbb{E}\Bigl(\Bigl|\int_{t_{k}}^{t}\tilde{f}(\bar{Y}(s))ds\Bigl|^{p}+\Bigl|
\sum_{j=1}^{m}\int_{t_{k}}^{t}\tilde{g}_{j}(\bar{Y}(s))dB^{j}(s)\Bigl|^{p}\\
&\quad+\Bigl|\sum_{j_{1}=1}^{m}\int_{t_{k}}^{t}\sum_{j_{2}=1}^{m}L^{j_{1}}\tilde{g}_{j_{2}}(\bar{Y}(s))\Delta B^{j_{2}}(s)dB^{j_{1}}(s)\Bigl|^{p}\Bigl)\\
\end{align*}
where and from now on $C$ is a positive constant independent of $\Delta$ that may change from line to line. Then by the elementary inequality, the H\"{o}lder inequality and Theorem 7.1 in \cite{AA} (page 39),

\begin{align*}
\mathbb{E}|Y(t)-\bar{Y}(t)|^{p}&\le C \Bigl(\Delta^{p-1}\mathbb{E}\int_{t_{k}}^{t}\Bigl|\tilde{f}(\bar{Y}(s))
\Bigl|^{p}ds+\Delta^{\frac{p}{2}-1}\sum_{j=1}^{m}\mathbb{E}\int_{t_{k}}^{t}\Bigl|\tilde{g}_{j}(\bar{Y}(s))\Bigl|^{p}ds\\
&+\Delta^{\frac{p}{2}-1}\mathbb{E}\sum_{j_{1}=1}^{m}\sum_{j_{2}=1}^{m}\int_{t_{k}}^{t}
\Bigl|L^{j_{1}}\tilde{g}_{j_{2}}(\bar{Y}(s))\Bigl|^{p}\Bigl|\Delta B^{j_{2}}(s)\Bigl|^{p}ds\Bigl).
\end{align*}

By Lemma \ref{A4} and the fact that $\mathbb{E}|\Delta B^{j_{2}}(s)|^{p}\le C\Delta^{\frac{p}{2}}$ for $s\in[t_{k},t_{k+1})$, we obtain
\begin{align}\label{B7}
\mathbb{E}|Y(t)-\bar{Y}(t)|^{p}&\le C\bigl(\bar{K}_{h(\Delta)}^{p}\Delta^{p}\mathbb{E}|Y_{k}|^{p}+\Delta^{p}+
\bar{K}_{h(\Delta)}^{p}\Delta^{\frac{p}{2}}\mathbb{E}|Y_{k}|^{p}+\Delta^{\frac{p}{2}}\bigl)\notag\\
&\le C\bigl(\bar{K}_{h(\Delta)}^{p}\Delta^{\frac{p}{2}}\mathbb{E}|Y_{k}|^{p}+\Delta^{\frac{p}{2}}\bigl).
\end{align}

It\^{o}'s formula yields
\begin{align*}
\mathbb{E}|Y(t)|^{p}& \le\mathbb{E}|Y(0)|^{p}+p\mathbb{E}\int_{0}^{t}|Y(s)|^{p-2}
\Bigl\langle Y(s),\tilde{f}(\bar{Y}(s))\Bigl\rangle ds\\
& \quad+p\sum_{j_{1}=1}^{m}\mathbb{E}\int_{0}^{t}\frac{p-1}{2}|Y(s)|^{p-2}
\Bigl|\tilde{g}_{j_{1}}(\bar{Y}(s))+\sum_{j_{2}=1}^{m}L^{j_{1}}\tilde{g}_{j_{2}}(\bar{Y}(s))\Delta B^{j_{2}}(s)\Bigl|^{2}ds.
\end{align*}

Since
$$p|Y(s)|^{p-2}\sum^{m}_{j_{1}=1}\Bigl\langle Y(s),\tilde{g}_{j_{1}}(\bar{Y}(s))+\sum_{j_{2}=1}^{m}L^{j_{1}}\tilde{g}_{j_{2}}(\bar{Y}(s))\Delta B^{j_{2}}(s)\Bigl\rangle$$ is obvious $\mathcal{F}_{s}$-measurable and
\begin{align*}
\mathbb{E}\Bigl(\sum_{j_{1}=1}^{m}\int_{0}^{t}p|Y(s)|^{p-2}\Bigl\langle Y(s),\tilde{g}_{j_{1}}(\bar{Y}(s))+\sum_{j_{2}=1}^{m}L^{j_{1}}\tilde{g}_{j_{2}}(\bar{Y}(s))\Delta B^{j_{2}}(s)\Bigl\rangle dB^{j_{1}}(s)\Bigl)=0.
\end{align*}

Then
\begin{align*}
\mathbb{E}|Y(t)|^{p}& \le\mathbb{E}|Y(0)|^{p}+p\mathbb{E}\int_{0}^{t}|Y(s)|^{p-2}\Bigl(\Bigl\langle \bar{Y}(s),\tilde{f}(\bar{Y}(s))\Bigl\rangle+\frac{p-1}{2}\sum_{j=1}^{m}\Bigl|\tilde{g}_{j}(\bar{Y}(s))\Bigl|^{2}\Bigl)ds\\
& +p(p-1)\sum_{j_{1}=1}^{m}\mathbb{E}\int_{0}^{t}|Y(s)|^{p-2}\Bigl|\sum_{j_{2}=1}^{m}L^{j_{1}}\tilde{g}_{j_{2}}(\bar{Y}(s))\Delta B^{j_{2}}(s)\Bigl|^{2}ds\\
&+ p\mathbb{E}\int_{0}^{t}|Y(s)|^{p-2}\Bigl\langle Y(s)-\bar{Y}(s),\tilde{f}(\bar{Y}(s))\Bigl\rangle ds.
\end{align*}

By (\ref{A4}) and (\ref{B3}), we have
\begin{align*}
\mathbb{E}|Y(t)|^{p}&\le \mathbb{E}|Y(0)|^{p}+C\mathbb{E}\int_{0}^{t}|Y(s)|^{p-2}\Bigl(1+|\bar{Y}(s)|^{2}\Bigl)ds\\
&\quad+C\mathbb{E}\int_{0}^{t}|Y(s)|^{p-2}(|\bar{Y}(s)|^{2}|\bar{K}_{h(\Delta)}|^{2}+1)\Delta ds\\&\quad+p\mathbb{E}\int_{0}^{t}|Y(s)|^{p-2}\Bigl\langle Y(s)-\bar{Y}(s),\tilde{f}(\bar{Y}(s))\Bigl\rangle ds.
\end{align*}

Moreover, Young's inequality implies
\begin{align}\label{B9}
\mathbb{E}|Y(t)|^{p}&\le \mathbb{E}|Y(0)|^{p}+C\mathbb{E}\int_{0}^{t}|Y(s)|^{p}ds+C\mathbb{E}\int_{0}^{t}|\bar{Y}(s)|^{p}ds+Ct\notag\\
&\quad+C\mathbb{E}\int_{0}^{t}|Y(s)-\bar{Y}(s)|^{\frac{p}{2}}|\tilde{f}(\bar{Y}(s))|^{\frac{p}{2}}ds.
\end{align}

On the other hand, by (\ref{B7}) and (\ref{A4}), we have
\begin{align*}
&\quad\mathbb{E}\int_{0}^{t}|Y(s)-\bar{Y}(s)|^{\frac{p}{2}}|\tilde{f}(\bar{Y}(s))|^{\frac{p}{2}}ds\\&\le C\mathbb{E}\int_{0}^{t}(\bar{K}_{h(\Delta)}^{\frac{p}{2}}|\bar{Y}(s)|^{\frac{p}{2}}+1)|Y(s)-\bar{Y}(s)|^{\frac{p}{2}}ds\\
& \le C\mathbb{E}\int_{0}^{t}|Y(s)-\bar{Y}(s)|^{\frac{p}{2}}ds
+C\mathbb{E}\int_{0}^{t}\bar{K}_{h(\Delta)}^{\frac{p}{2}}|Y(s)-\bar{Y}(s)|^{\frac{p}{2}}\cdot |\bar{Y}(s)|^{\frac{p}{2}}ds\\
&\le CT+C(1+\bar{K}_{h(\Delta)}^{p})\mathbb{E}\int_{0}^{t}|Y(s)-\bar{Y}(s)|^{p}ds+C\mathbb{E}\int_{0}^{t}|\bar{Y}(s)|^{p}ds\\
&\le CT+C\mathbb{E}\int_{0}^{t}|\bar{Y}(s)|^{p}ds+C(1+\bar{K}_{h(\Delta)}^{p})
\int_{0}^{t}\bigl(\Delta^{\frac{p}{2}}+\bar{K}_{h(\Delta)}^{p}\Delta^{\frac{p}{2}}\mathbb{E}|\bar{Y}(s)|^{p}\bigl)ds
\end{align*}

 By Lemma \ref{l0} and choose $\Delta\le 1$ small enough, we have
\begin{align*}
(1+K_{h(\Delta)}^{p})(K_{h(\Delta)}^{p}\Delta^{\frac{p}{2}}+\bar{K}_{h(\Delta)}^{p}\Delta^{p})\le (1+\bar{K}_{h(\Delta)}^{p/2})(\bar{K}_{h(\Delta)}^{p/2}\Delta^{p/2}+\bar{K}_{h(\Delta)}^{p}\Delta^{p})\le 2.
\end{align*}

Thus we have
\begin{align}\label{B10}
\mathbb{E}\int_{0}^{t}|Y(s)-\bar{Y}(s)|^{\frac{p}{2}}|\tilde{f}(\bar{Y}(s))|^{\frac{p}{2}}ds\le CT+C\mathbb{E}\int_{0}^{t}|\bar{Y}(s)|^{p}ds.
\end{align}

Substituting (\ref{B10}) into (\ref{B9}), by (\ref{B8}) we have
\begin{align*}
\mathbb{E}|Y(t)|^{p}\le\mathbb{E}|Y(0)|^{p}+CT+C\int_{0}^{t}\Bigl(\sup_{0\le u \le s}\mathbb{E}|Y(u)|^{p}\Bigl)ds.
\end{align*}

By the Gronwall inequality, we obtain
\begin{align*}
\sup_{0\le s \le t}\mathbb{E}|Y(s)|^{p}\le C(1+\mathbb{E}|Y(0)|^{p}).
\end{align*}

Since $C$ is independent of $\Delta$. The proof is complete.
\end{proof}
\end{Lemma}

For any real number $R>|x(0)|$, we define two stopping times
\begin{align*}
\tau_{R}:=\inf\{t\ge 0,|x(t)|\ge R\} \ and\ \rho_{R}:=\inf\{t\ge 0,|Y(t)|\ge R\}.
\end{align*}
\begin{Lemma}\label{A11}
Let Assumption $\ref{A2}$ and $(\ref{BA31})$ hold. For any real number $R>|x_{0}|$, the estimate
\begin{align*}
\mathbb{P}(\tau_{R}\le T)\le \frac{C}{R^{p}}
\end{align*}
holds for some positive constant $C$ independent of $R$.
\begin{proof}
The proof of this lemma is similar to that of (\ref{B16}). In other words, replacing $t$ by $\tau_{R}\land T$ in (\ref{B16}), we have
\begin{align*}
\mathbb{E}|x(\tau_{R}\land T)|^{p}\le C.
\end{align*}

Then
\begin{align*}
C\ge\mathbb{E}|x(\tau_{R}\land T)|^{p}\ge\mathbb{E}|x(t)|^{p}1_{\{\tau_{R}\le T\}}\ge R^{p}\mathbb{P}(\tau_{R}\le T),
\end{align*}
which implies the assertion.
\end{proof}
\end{Lemma}

\begin{Lemma}\label{A21}
Let Assumption $\ref{A2}$ and $(\ref{BA31})$ hold. For any real number $R>|x(0)|$ and any sufficiently small $\Delta\in(0,\Delta^{\ast}]$, it follows that
\begin{align*}
\mathbb{P}(\rho_{R}\le T)\le \frac{C}{R^{p}}.
\end{align*}
\end{Lemma}
The proof is similar to that of Lemma \ref{A11}, so we omit it here.

\begin{Lemma}\label{D2}
Let $Z$ be a predictable stochastic process satisfying $\mathbb{P}(\int^{T}_{0}|Z_{s}|^{2}ds<\infty)=1$. Then for all $t\in[0,T]$ and all $p\ge 2$
\begin{equation}
\Bigl\|\sup_{[0,t]}\Bigl|\int_{0}^{s}Z_{u}dB_{u}\Bigl|\Bigl\|_{L^{p}}\le p\Bigl(\int_{0}^{t}\|Z_{s}\|^{2}_{L^{p}}ds\Bigl)^{1/2}
\end{equation}
where $||X||_{L^p}:=(\mathbb{E}|X|^p)^{1/p}.$
\begin{proof}
Denote $A_{t}=\int_{0}^{t}Z_{s}dB_{s}$, by the It\^{o}'s formula, we have
\begin{align*}
|A(t)|^{p}&=\frac{p}{2}\int_{0}^{t}|A(s)|^{p-2}\cdot |Z_{s}|^{2}ds+p\int_{0}^{t}|A(s)|^{p-2}\cdot \langle A(s),Z_{s} dB_{s}\rangle\\
&\quad+\frac{p(p-2)}{2}\int_{0}^{t}|A(s)|^{p-4}\cdot|A_{s}^{\ast}\cdot Z_{s}|^{2}ds\\
&\le \frac{p(p-1)}{2}\int_{0}^{t}|A(s)|^{p-2}\cdot |Z_{s}|^{2}ds+p\int_{0}^{t}|A(s)|^{p-2}\cdot \langle A(s),Z_{s} dB_{s}\rangle.
\end{align*}

It is easy to know that the second term on the right-hand side is martingale. So we take expectation in the inequality and use H\"{o}lder inequality to derive
\begin{align*}
\mathbb{E}|A(t)|^{p}&\le\frac{p(p-1)}{2}\mathbb{E}\int_{0}^{t}|A(s)|^{p-2}\cdot |Z_{s}|^{2}ds\\
&\le \frac{p(p-1)}{2}\int_{0}^{t}(\mathbb{E}|A(s)|^{p})^{(p-2)/p}\cdot (\mathbb{E}|Z_{s}|^{p})^{2/p}ds\\
&\le \frac{p(p-1)}{2}\int_{0}^{t}(\sup_{u\in[0,s]}\mathbb{E}|A(u)|^{p})^{(p-2)/p}\cdot (\mathbb{E}|Z_{s}|^{p})^{2/p}ds\\
&\le \frac{p(p-1)}{2}(\sup_{s\in[0,t]}\mathbb{E}|A(s)|^{p})^{(p-2)/p}
\int_{0}^{t}(\mathbb{E}|Z_{s}|^{p})^{2/p}ds.
\end{align*}

Since the right hand side is increasing in $t$, it follows that
\begin{align*}
\sup_{s\in[0,t]}\mathbb{E}|A(s)|^{p}&\le \frac{p(p-1)}{2}(\sup_{s\in[0,t]}\mathbb{E}|A(s)|^{p})^{(p-2)/p}
\int_{0}^{t}\|Z_{s}\|^{2}_{L^{p}}ds.
\end{align*}
Dividing both sides by $(\sup_{s\in[0,t]}\mathbb{E}|A(s)|^{p})^{(p-2)/p}$ and applying Doob's maximal inequality one gets
\begin{align*}
\mathbb{E}(\sup_{s\in[0,t]}|A(s)|^{p})&\le (\frac{p}{p-1})^{p}\mathbb{E}|A(t)|^{p}\\
&\le(\frac{p}{p-1})^{p}\cdot(\frac{p(p-1)}{2})^{p/2}(\int_{0}^{t}\|Z_{s}\|^{2}_{L^{p}}ds)^{p/2}
\end{align*}
By the fact that $(\frac{p}{p-1})\cdot(\frac{p(p-1)}{2})^{1/2}\le p$ if $p\ge 2$, we complete the proof.
\end{proof}
\end{Lemma}
\begin{Lemma}\label{D1}
Let $Z_{1},...,Z_{N}:\Omega\rightarrow \mathbb{R}$ be $\mathcal{F}/\mathcal{B}(\mathbb{R})$-measurable mapping with $\mathbb{E}\|Z_{n}\|^{p}<\infty$ for all $n\in\{1,...,N\}$ and with $\mathbb{E}[Z_{n+1}|Z_{1},...,Z_{n}]=0$ for all $n\in\{1,...,N\}$. Then
\begin{equation}
\|Z_{1}+...+Z_{n}\|_{L^{p}}\le c_{p}\Bigl(\|Z_{1}\|_{L^{p}}^{2}+...+\|Z_{n}\|_{L^{p}}^{2}\Bigl)^{\frac{1}{2}}
\end{equation}
for every $p\in[2,\infty)$, where $c_{p}$ are constants dependent on $p$, but independent of $n$.
\end{Lemma}
We can find the proof from Corollary 6.3.6 in \cite{AJ}.

\section{Strong convergence of the modified truncated Milstein method}

Let $\psi:\mathbb{R}^{d}\rightarrow\mathbb{R}^{d}$ be twice differentiable. Then by Taylor formula, we have
\begin{equation}\label{B11}
\psi(x)-\psi(y)=\psi'(y)(x-y)+R_{1}(\psi)(x,y),
\end{equation}
where $R_{1}(\psi)$ is the remainder term
\begin{equation}
R_{1}(\psi)(x,y):=\int_{0}^{1}(1-\varsigma)\psi''(x+\varsigma(x-y))(x-y,x-y)d\varsigma,
\end{equation}
and
\begin{equation}\label{B12}
\psi'(x)(h_{1}):=\sum_{i=1}^{d}\frac{\partial\psi(x)}{\partial x^{i}}h^{i}_{1},\quad\psi''(x)(h_{1},h_{2}):=\sum_{i,j=1}^{d}\frac{\partial^{2}\psi(x)}{\partial x^{i}\partial x^{j}}h^{i}_{1}h^{j}_{2}
\end{equation}
for any $x,h_{1},h_{2}\in\mathbb{R}^{d}$.
Here,
\begin{align*}
\psi=(\psi_{1},\psi_{2},...,\psi_{d})^T, \frac{\partial\psi}{\partial x^{i}}=\Bigl(\frac{\partial\psi_{1}}{\partial x^{i}},\frac{\partial\psi_{2}}{\partial x^{i}},...,\frac{\partial\psi_{d}}{\partial x^{i}}\Bigl)^T, \frac{\partial^{2}\psi(x)}{\partial x^{i}\partial x^{j}}=(\frac{\partial^{2}\psi_1(x)}{\partial x^i\partial x^j},\cdots,\frac{\partial^2\psi_d(x)}{\partial x^{i}\partial x^{j}})^T.
\end{align*}
Replacing $x$ and $y$ in (\ref{B11}) by $Y(t)$ and $\bar{Y}(t)$, respectively, by (\ref{B4}) we have
\begin{equation}\label{B15}
\psi(Y(t))-\psi(\bar{Y}(t))=\psi'(\bar{Y}(t))(\sum_{j=1}^{m}\int_{t_{k}}^{t}\tilde{g_{j}}(\bar{Y}(s))dB^{j}(s))+\tilde{R}_{1}(\psi),
\end{equation}
where
\begin{equation}\label{B18}
\aligned\tilde{R}_{1}(\psi)&=\psi'(\bar{Y}(t))\Bigl(\int_{t_{k}}^{t}\tilde{f}(\bar{Y}(s))ds+
\sum_{j_{1}=1}^{m}\int_{t_{k}}^{t}\sum_{j_{2}=1}^{m}L^{j_{1}}\tilde{g}_{j_{2}}(\bar{Y}(s))\Delta B^{j_{2}}(s)dB^{j_{1}}(s)\Bigl)\\&\quad+R_{1}(\psi)(Y(t),\bar{Y}(t)).\endaligned
\end{equation}

Notice that $\tilde{R}_{1}$ also depends on $\tilde{f}, \tilde{g}$ and $t$.

Now we have the following

\begin{Lemma}\label{l1}
For any $p\ge1$, $i=1,2,...,m$, let $\Delta\in(0,1)$ sufficiently small. If $R\le h(\Delta)$ and $t\le \theta:=\tau_{R}\land\rho_{R}$, then
\begin{equation}
\|\tilde{R}_{1}(f)\|_{L^{p}}\vee\|\tilde{R}_{1}(g_{i})\|_{L^{p}}\le CK_{R}\bar{K}_{R}^{2}\Delta.
\end{equation}
\begin{proof}
Firstly, we have an estimate on $\|R_{1}(f)\|_{L^{p}}$. Without loss of generality, suppose $\bar{K}_{R}\ge 1$. By Assumption \ref{A3}, $(\ref{BA31})$ and Burkholder-Davis-Gundy inequality, we can find a constant $C$ (independent of $R$ and $\Delta$) such that
\begin{align*}
&\quad\|R_{1}(f)(Y(t),\bar{Y}(t))\|_{L^{p}}\\&\le\int_{0}^{1}(1-\varsigma)
\|f''(\xi)(Y(t)-\bar{Y}(t),Y(t)-\bar{Y}(t))\|_{L^{p}}d\varsigma\\
&\le \int_{0}^{1}(1-\varsigma)\sum_{i,j}^{d}\Bigl\|\frac{\partial^{2}f(\xi)}{\partial x^{i}\partial x^{j}}\Bigl\|_{L^{2p}}\cdot \Bigl\|(Y(t)-\bar{Y}(t))_{i}(Y(t)-\bar{Y}(t))_{j}\Bigl\|_{L^{2p}}d\varsigma\\
&\le \int_{0}^{1}\sum_{i,j}^{d}\Bigl\|\sum_{l=1}^{d}\Bigl|\frac{\partial^{2}f_{l}(\xi)}{\partial x^{i}\partial x^{j}}\Bigl|^{2}\Bigl\|_{L^{p}}^{1/2}\cdot\Bigl\|(Y(t)-\bar{Y}(t))_{i}\Bigl\|_{L^{4p}
}\cdot\Bigl\|(Y(t)-\bar{Y}(t))_{j}\Bigl\|_{L^{4p}}d\varsigma\\
&\le \sqrt{d}K_{R}\cdot d^{2}\|Y(t)-\bar{Y}(t)\|_{L^{4p}}^{2}\\
&\le \sqrt{d}K_{R}\cdot d^{2}\cdot\Bigl(\int_{\kappa(t)}^{t}\|\tilde{f}(\bar{Y}(r))\|_{L^{4p}}dr
+\sum_{j=1}^{d}\Bigl\|\int_{\kappa(t)}^{t}\tilde{g}_{j}(\bar{Y}(r))dB_{r}^{j}\|_{L^{4p}}\\
&\quad+\sum_{j_{1}=1}^{d}\sum_{j_{2}=1}^{d}\Bigl\|\int_{\kappa(t)}^{t}
L^{j_{1}}\tilde{g}_{j_{2}}(\bar{Y}(r))\Delta B^{j_{2}}_{r}dB_{r}^{j}\Bigl\|_{L^{4p}}\Bigl)^{2}\\
&\le CK_{R}(\Delta^2+\Delta) (\bar{K}^2_{R}\|\bar{Y}_{r}\|^2_{L^{4p}}+1)\\
&\le CK_{R}\bar{K}_{R}^{2}\Delta
\end{align*}
where we denote $[\frac{t}{\Delta}]$ by $\kappa(t)$ and $\xi=\bar{Y}(t)+\varsigma(Y(t)-\bar{Y}(t))$.
Notice that we have used the condition $t\le\theta$ in the forth inequality. By (\ref{B18}), it follows that
\begin{align*}
\|\tilde{R}_{1}(f)\|_{L^{p}}&\le \|f'(\bar{Y}(t))\int_{\kappa(t)}^{t}\tilde{f}(\bar{Y}(r))dr\|_{L^{p}}
+\|R_{1}(f)(Y(t),\bar{Y}(t))\|_{L^{p}}\\
&\quad+\|f'(\bar{Y}(t))\sum_{j_{1}=1}^{m}
\sum_{j_{2}=1}^{m}\int_{\kappa(t)}^{t}L^{j_{1}}g_{j_{2}}(\bar{Y}(r))\Delta B^{j_{2}}(r)dB^{j_{1}}(r)\|_{L^{p}}\\
&\le \|f'(\bar{Y}(t))\|_{L^{2p}}\cdot\|\int_{\kappa(t)}^{t}\tilde{f}(\bar{Y}(r))dr\|_{L^{2p}}
+CK_{R}\bar{K}_{R}^{2}\Delta\\
&\quad+\|f'(\bar{Y}(t))\|_{L^{2p}}\cdot\|\sum_{j_{1}=1}^{m}
\sum_{j_{2}=1}^{m}\int_{\kappa(t)}^{t}L^{j_{1}}g_{j_{2}}(\bar{Y}(r))\Delta B^{j_{2}}(r)dB^{j_{1}}(r)\|_{L^{2p}}\\
&\le CK_{R}\bar{K}_{R}\Delta+CK_{R}\bar{K}_{R}^{2}\Delta+CK_{R}\bar{K}_{R}\Delta\\&
\le CK_{R}\bar{K}_{R}^{2}\Delta.
\end{align*}

Similarly, one derive that
\begin{align*}
\|\tilde{R}_{1}(g_{i})\|_{L^{p}}\le CK_{R}\bar{K}_{R}^{2}\Delta.
\end{align*}
\end{proof}
\end{Lemma}

\begin{Theorem}\label{A12}
Let Assumption $\ref{A3}$ and $\ref{A2}$ hold for $p\ge4$. For any fixed $R>|x_{0}|$ (sufficiently large), if $\Delta\in(0,\Delta^{\ast}]$ is chosen to be sufficiently small such that $h(\Delta)\ge R$, then
\begin{equation}
\|\sup_{0\le s\le t}e(s\land\theta)\|_{L^{p}}\le C\bar{K}^{5/2}_{h(\Delta)}\Delta
\end{equation}
where $\theta:=\tau_{R}\land\rho_{R}$ and $e(t):=x(t)-Y(t)$.

\begin{proof}
Combining (\ref{B1}) and (\ref{B4}) gives
\begin{equation}\label{B21}
\aligned x(t)-Y(t)&=\int_{0}^{t}[f(x(s))-\tilde{f}(\bar{Y}(s))]ds+\sum_{i=1}^{m}\int_{0}^{t}D_i(s)dB^{i}(s)\endaligned
\end{equation}
where
$$D_i(s):=g_{i}(x(s))-
\tilde{g}_{i}(\bar{Y}(s))-\sum_{j=1}^{m}L^{j}\tilde{g}_{i}(\bar{Y}(s))\Delta B^{j}(s).$$
Applying It\^{o} formula to (\ref{B21}) gives
\begin{equation}\label{B22}
\aligned  |e(t\land\theta)|^{2}&=2\int_{0}^{t\land\theta}\Bigl\langle e(s),f(x(s))-\tilde{f}(\bar{Y}(s))\Bigl\rangle ds+2\sum_{i=1}^{m}\int_{0}^{t\land\theta}\Bigl\langle e(s),D_i(s)\Bigl\rangle dB^{i}(s)\\
&\quad+\sum_{i=1}^{m}\int_{0}^{t\land\theta}\Bigl|D_i(s)\Bigl|^{2}ds.\endaligned
\end{equation}

Since $0\le s\le t\land\theta$, we have $|Y(s)|<R\le h(\Delta)$, which yields $|\bar{Y}(s)|<h(\Delta)$. According to (\ref{mtf1}), (\ref{mtf2}) and (\ref{mtf3}), we have
\begin{align*}
\tilde{f}(\bar{Y}(s))=f(\bar{Y}(s)),\  \ \tilde{g}_{i}(\bar{Y}(s))=g_{i}(\bar{Y}(s))\ \textrm{and}\ L^{j}\tilde{g}_{i}(\bar{Y}(s))=L^{j}g_{i}(\bar{Y}(s)).
\end{align*}

For the integrand of the first term in (\ref{B22}), we use $(\ref{BA31})$ to arrive at
\begin{equation}\label{C1}
\aligned\Bigl\langle e(s),f(x(s))-\tilde{f}(\bar{Y}(s))\Bigl\rangle&\le |e(s)|\cdot|f(x(s))-f(Y(s))|
+\Bigl\langle e(s),f(Y(s))-f(\bar{Y}(s))\Bigl\rangle\\
&\le \bar{K}_{R}|e(s)|^{2}+\Bigl\langle e(s),f(Y(s))-f(\bar{Y}(s))\Bigl\rangle.\endaligned
\end{equation}

Notice that
\begin{align*}
g'_{i}(\bar{Y}(s))(\sum_{j=1}^{m}\int_{\kappa(s)}^{s}\tilde{g}_{j}(\bar{Y}(r))dB^{j}(r))&
=\sum_{l=1}^{d}\frac{\partial g_{i}(\bar{Y}(s))}{\partial x^{l}}\left(\sum_{j=1}^{m}\tilde{g}_{l,j}(\bar{Y}(s))\Delta B^{j}_{s}\right)\\
&=\sum_{j=1}^{m}\sum_{l=1}^{d}\tilde{g}_{l,j}(\bar{Y}(s))\frac{\partial g_{i}}{\partial x^{l}}(\bar{Y}(s))\Delta B^{j}_{s}\\&
=\sum_{j=1}^{m}L^{j}g_{i}(\bar{Y}(s))\Delta B^{j}(s).
\end{align*}

Then one can use an elementary inequality, the notation (\ref{B15}) to get
\begin{equation}\label{C2}
\aligned
 \Bigl|D_i(s)\Bigl|^{2}&=\Bigl|g_{i}(x(s))-g_{i}(Y(s))+\tilde{R}_{1}(g_{i})\Bigl|^{2}\\
&\le 2|g_{i}(x(s))-g_{i}(Y(s))|^{2}+2|\tilde{R}_{1}(g_{i})|^{2}
\\&\le 2\bar{K}^{2}_{R}|e(s)|^{2}+2|\tilde{R}_{1}(g_{i})|^{2}.
\endaligned
\end{equation}

Inserting (\ref{C1}) and (\ref{C2}) into (\ref{B22}) yields
\begin{equation}\aligned
|e(t\land\theta)|^{2}&\le(2\bar{K}_{R}+2m\bar{K}_{R}^{2})\int_{0}^{t\land\theta}|e(s)|^{2}ds
+\int_{0}^{t\land\theta}\Bigl\langle e(s),f(Y(s))-f(\bar{Y}(s))\Bigl\rangle ds\notag\\
&\quad+2\sum_{i=1}^{m}\int_{0}^{t\land\theta}|\tilde{R}_{1}(g_{i})|^{2}ds+2\sum_{i=1}^{m}\int_{0}^{t\land\theta}\Bigl\langle e(s),D_i(s)\Bigl\rangle dB^{i}(s).
\endaligned
\end{equation}

Since $p\ge 4$, we have
\begin{align}\label{C3}
&\quad\Bigl\|\sup_{0\le s\le t}|e(s\land\theta)|\Bigl\|^{2}_{L^{p}}=\Bigl\|\sup_{0\le s\le t}|e(s\land\theta)|^{2}\Bigl\|_{L^{\frac{p}{2}}}\notag\\
&\le (2\bar{K}_{R}+2m\bar{K}_{R}^{2})\int_{0}^{t}\|e(s\land\theta)\|^{2}_{L^{p}}ds
+2\sum_{i=1}^{m}\int_{0}^{t}\|\tilde{R}_{1}(g_{i})\|^{2}_{L^{p}}ds\notag\\
&\quad+2\|\sup_{0\le s\le t}\int_{0}^{s\land\theta}\Bigl\langle e(u),f(Y(u))-f(\bar{Y}(u))\Bigl\rangle du\|_{L^{\frac{p}{2}}}\notag\\
&\quad+2\|\sup_{0\le s\le t}\sum_{i=1}^{m}\int_{0}^{s\land\theta}\Bigl\langle e(u),D_i(u)\Bigl\rangle dB^{i}(u)\|_{L^{\frac{p}{2}}}.
\end{align}

For the last term on the right-hand side of (\ref{C3}), we use Lemma \ref{D2} and (\ref{BA31}) to derive
\begin{align}\label{C4}
&\quad 2\|\sup_{0\le s\le t}\sum_{i=1}^{m}\int_{0}^{s\land\theta}\Bigl\langle e(u),D_i(u)\Bigl\rangle dB^{i}(u)\|_{L^{\frac{p}{2}}}\notag\\
&\le p\sum_{i=1}^{m}\Bigl( \int_{0}^{t}\Bigl\|\Bigl\langle e(s\land\theta),D_i(u\land\theta)\Bigl\rangle
\Bigl\|^{2}_{L^{\frac{p}{2}}}ds\Bigl)^{1/2}\notag\\
& \le \sup_{0\le s\le t}\|e(s\land\theta)\|_{L^{p}}\cdot p\sum_{i=1}^{m}\Bigl(\int_{0}^{t}\|D_i(u\land\theta)\|^{2}_{L^{p}}ds\Bigl)^{1/2}\notag\\
& \le\frac{1}{4}\sup_{0\le s\le t}\|e(s\land\theta)\|^{2}_{L^{p}}+2p^{2}m\bar{K}^{2}_{R}
\sum_{i=1}^{m}\int_{0}^{t}\|e(s\land\theta)\|^{2}_{L^{p}}ds\notag\\
&\quad+2p^{2}m\sum_{i=1}^{m}\int_{0}^{t}\|\tilde{R}_{1}(g_{i})\|^{2}_{L^{p}}ds.
\end{align}

At the same time, denoting $\kappa(t):=[\frac{t}{\Delta}]\Delta$, $$H(u):=f'(\bar{Y}(u))\sum_{j=1}^{m}\int_{\kappa(u)}^{u}\tilde{g}_{j}
(\bar{Y}(r))dB^{j}(r)$$ and replacing $\phi$ in (\ref{B15}) by $f$. We have
\begin{align}\label{C5}
&\quad2\Bigl\|\sup_{0\le s\le t}\int_{0}^{s\land\theta}\Bigl\langle e(u),f(Y(u))-f(\bar{Y}(u))\Bigl\rangle du\Bigl\|_{L^{p/2}}\notag\\
&=2\Bigl\|\sup_{0\le s\le t}\int_{0}^{s\land\theta}\Bigl\langle e(u),H(u)+\tilde{R}_{1}(f)\Bigl\rangle du\Bigl\|_{L^{p/2}}\notag\\
&\le J+2\Big\|\sup_{0\le s\le t}\int_{0}^{s\land\theta}\Bigl\langle e(u),\tilde{R}_{1}(f)\Bigl\rangle du\Bigl\|_{L^{p/2}}\notag\\
&\le J+\int_{0}^{t}\|e(s\land\theta)\|^{2}_{L^{p}}ds
+\int_{0}^{t}\|\tilde{R}_{1}(f)\|^{2}_{L^{p}}ds
\end{align}
where
\begin{equation}
J=2\Bigl\|\sup_{0\le s\le t}\int_{0}^{s\land\theta}\Bigl\langle e(u),H(u)\Bigl\rangle du\Bigl\|_{L^{p/2}}.
\end{equation}

Inserting (\ref{C4}) and (\ref{C5}) to (\ref{C3}) yields
\begin{align}
\frac{3}{4}\Bigl\|\sup_{0\le s\le t}|e(s\land\theta)|\Bigl\|^{2}_{L^{p}}&\le
2(\bar{K}_{R}+m\bar{K}_{R}^{2}+1+p^{2}m^{2}\bar{K}^{2}_{R})
\int_{0}^{t}\|e(s\land\theta)\|^{2}_{L^{p}}ds\notag\\
&\quad+2(1+p^{2}m)\sum_{i=1}^{m}\int_{0}^{t}
\|\tilde{R}_{1}(g_{i})\|^{2}_{L^{p}}ds+\int_{0}^{t}\|\tilde{R}_{1}(f)\|^{2}_{L^{p}}ds+J.
\end{align}

Therefore, it remains to estimate $J$. For $u\in (0, s\land\theta]$, (\ref{B1}), (\ref{B4}) and (\ref{B15}) implies
\begin{align}
e(u)&=e(\kappa(u))+\int_{\kappa(u)}^{u}[f(x(r))-\tilde{f}(\bar{Y}(r))]dr+\sum_{i=1}^{m}\int_{\kappa(u)}^{u}D_i(r)dB^{i}(r).\notag\\
&=\int_{\kappa(u)}^{u}[f(x(r))-f(x(\kappa(r))]dr+\sum_{i=1}^{m}\int_{\kappa(u)}^{u}[g_{i}(x(r))-g_{i}(Y(r))]dB^{i}(r)\notag\\
&\quad+\sum_{i=1}^{m}\int_{\kappa(u)}^{u}\tilde{R}_{1}(g_{i})dB^{i}(r)+(u-\kappa(u))
(f(x(\kappa(u)))-f(Y(\kappa(u)))+e(\kappa(u)).
\end{align}

Thus,
\begin{align*}
 J&\le2\Bigl\|\sup_{0\le s\le t}\int_{0}^{s\land\theta}\Bigl\langle \int_{\kappa(u)}^{u}[f(x(r))-f(x(\kappa(r))]dr,H(u)\Bigl\rangle du\Bigl\|_{L^{p/2}}\\
&\quad+2\Bigl\|\sup_{0\le s\le t}\int_{0}^{s\land\theta}\Bigl\langle \sum_{i=1}^{m}\int_{\kappa(u)}^{u}[g_{i}(x(r))-g_{i}(Y(r))]dB^{i}(r),H(u)\Bigl\rangle du\Bigl\|_{L^{p/2}}\\
&\quad+2\Bigl\|\sup_{0\le s\le t}\int_{0}^{s\land\theta}\Bigl\langle \sum_{i=1}^{m}\int_{\kappa(u)}^{u}\tilde{R}_{1}(g_{i})dB^{i}(r),H(u)\Bigl\rangle du\Bigl\|_{L^{p/2}}\\
&\quad+2\Bigl\|\sup_{0\le s\le t}\int_{0}^{s\land\theta}\Bigl\langle (u-\kappa(u))(f(x(\kappa(u)))-f(Y(\kappa(u))),H(u)\Bigl\rangle du\Bigl\|_{L^{p/2}}\\
&\quad+2\Bigl\|\sup_{0\le s\le t}\int_{0}^{s\land\theta}\Bigl\langle x(\kappa(u))-Y(\kappa(u)),H(u)\Bigl\rangle du\Bigl\|_{L^{p/2}}\\
&=: J_{1}+J_{2}+J_{3}+J_{4}+J_{5}.
\end{align*}

To begin our estimate, we show that
\begin{align}\label{C6}
&\quad\|x(t\land\theta)-x(\kappa(t\land\theta))\|_{L^{p}}\notag\\&
\le\int_{\kappa(t)}^{t}\|f(x(s\land\theta))\|_{L^{p}}ds
+\|\sum_{j=1}^{m}\int_{\kappa(t\land\theta)}^{t\land\theta}
g_{j}(x(s))dB^{j}(s)\|_{L^{p}}\notag\\
&\le\int_{\kappa(t)}^{t}(\bar{K}_{R}\|x(s\land\theta)\|_{L^{p}}
+\|f(0)\|_{L^{p}})ds
+\sum_{j=1}^{m}(\mathbb{E}|\int_{\kappa(t\land\theta)}^{t\land\theta}g_{j}(x(s))dB^{j}(s)|^{p})^{1/p}\notag\\
&\le C\bar{K}_{R}\Delta+C\sum_{j=1}^{m}\Bigl(\mathbb{E}\bigl(\int_{\kappa(t\land\theta)}^{t\land\theta}
|g_{j}(x(s))|^{2}ds\bigl)^{p/2}\Bigl)^{1/p}\notag\\
&\le C\bar{K}_{R}\Delta+C\sum_{j=1}^{m}\Bigl(\Delta^{\frac{p}{2}-1}\mathbb{E}\int_{\kappa(t\land\theta)}^{t\land\theta}|g_{j}(x(s))|^{p}ds\Bigl)^{1/p}\notag\\
&\le C\bar{K}_{R}\Delta+C\bar{K}_{R}\Delta^{\frac{1}{2}}\notag\\&
\le C\bar{K}_{R}\Delta^{\frac{1}{2}}.
\end{align}

Besides, for $u\le s\land\theta$, we can derive
\begin{align}\label{C7}
\|H(u)\|_{L^{p}}&\le K_{R}\|\sum_{j=1}^{m}\int_{\kappa(u)}^{u}\tilde{g}_{j}(\bar{Y}(r)))dB^{j}(r))\|_{L^{p}}\le CK_{R}\bar{K}_{R}\Delta^{\frac{1}{2}}.
\end{align}

For $J_{1}$, by (\ref{C6}), (\ref{C7}) and H\"{o}lder inequality, we have
\begin{align*}
J_{1}&\le 2\int_{0}^{t}\int_{\kappa(u)}^{u}\|f(x(r\land\theta))-f(x(\kappa(r\land\theta))\|_{L^{p}}
\cdot\|H(u)\|_{L^{p}}drdu\\
&\le2\int_{0}^{t}\int_{\kappa(u)}^{u}C\bar{K}^{2}_{R}\Delta^{\frac{1}{2}}\cdot CK_{R}\bar{K}_{R}\Delta^{\frac{1}{2}}du\\
&\le CK_{R}\bar{K}_{R}^{3}\Delta^{2}.
\end{align*}

For the term $J_{2}$, we use an elementary inequality, Lemma \ref{D1}, (\ref{C7}) and H\"{o}lder inequality to get
\begin{align*}
J_{2}&\le 2\int_{0}^{t}\|\sum_{i=1}^{m}\int_{\kappa(u\land\theta)}^{u\land\theta}[g_{i}(x(r))-g_{i}(Y(r))]dB^{i}(r) \|_{L^{p}}\cdot\|H(u\land\theta)\|_{L^{p}}du\\
&\le \frac{1}{\Delta}\int_{0}^{t}\|\sum_{i=1}^{m}
\int_{\kappa(u\land\theta)}^{u\land\theta}[g_{i}(x(r))-g_{i}(Y(r))]dB^{i}(r) \|^{2}_{L^{p}}du+\Delta\int_{0}^{t}\|H(u\land\theta)\|^{2}_{L^{p}}du\\
&\le\frac{p^{2}}{\Delta}\int_{0}^{t}\int_{\kappa(u)}^{u}
\sum_{i=1}^{m}\|g_{i}(x(r\land\theta))-g_{i}(Y(r\land\theta))\|^{2}_{L^{p}}drdu
+CK_{R}^{2}\bar{K}_{R}^{2}\Delta^{2}\\
&\le mp^{2}\bar{K}^{2}_{R}\int_{0}^{t}\sup_{0\le r\le u}\|e(r\land\theta)\|^{2}_{L^{p}}du+CK_{R}^{2}\bar{K}_{R}^{2}\Delta^{2}.
\end{align*}

By Lemma \ref{D1} and Lemma \ref{l1}, we derive that similarly as above
\begin{align*}
J_{3}&\le 2\int^{t}_{0}\|\sum_{i=1}^{m}\int_{\kappa(u\land\theta)}^{u\land\theta}\tilde{R}_{1}(g_{i})dB^{i}(r)
\|_{L^{p}}\cdot \|H(u\land\theta)\|_{L^{p}}du\\
&\le CK_{R}\bar{K}_{R}^{2}\Delta^{3/2}\cdot CK_{R}\bar{K}_{R}\Delta^{1/2}\\
&\le CK_{R}^{2}\bar{K}_{R}^{3}\Delta^{2}.
\end{align*}

For $J_{4}$, it's a little different as follows, and we always suppose that $u\le t\land\theta$ and denote $N_t=[t/\Delta]$ in the sequel
\begin{align*}
J_{4}&\le
2\Bigl\| \sup_{0\le s\le t}\Bigl|\sum_{k=0}^{N_s-1}\int_{t_{k}}^{t_{k+1}}\Bigl\langle (u-\kappa(u))(f(x(\kappa(u)))-f(Y(\kappa(u)))),H(u)\Bigl\rangle du\Bigl|\Bigl\|_{L^{p/2}}\\
&\quad+2\Bigl\| \sup_{0\le s\le t}\Bigl|\int_{\kappa(s)}^{s}\Bigl\langle (u-\kappa(u))(f(x(\kappa(u)))-f(Y(\kappa(u)))),H(u)\Bigl\rangle du\Bigl|\Bigl\|_{L^{p/2}}\\
&=: J_{41}+J_{42}.
\end{align*}

 It is easy to show that $Z_k:=\int_{t_{k}}^{t_{k+1}}\Bigl\langle (u-\kappa(u))(f(x(\kappa(u)))-f(Y(\kappa(u)))),H(u)\Bigl\rangle du$ satisfies $\mathbb{E}[Z_{k+1}|Z_{1},...,Z_{k}]=0$. For $p\ge 4$, applying Lemma \ref{D1} and H\"{o}lder inequality, we can get
\begin{align*}
J_{41}&\le\frac{2p}{p-2}\Bigl\|\sum_{k=0}^{N_{t}-1}\int_{t_{k}}^{t_{k+1}}\Bigl\langle (u-\kappa(u))(f(x(\kappa(u)))-f(Y(\kappa(u)))),H(u)\Bigl\rangle du\Bigl\|_{L^{p/2}}\\
&\le \frac{2pC_{p/2}}{p-2}\Bigl(\sum_{k=0}^{N_{t}-1}\Bigl\|\int_{t_{k}}^{t_{k+1}}\Bigl\langle (u-\kappa(u))(f(x(\kappa(u)))-f(Y(\kappa(u)))),H(u)\Bigl\rangle du\Bigl\|^{2}_{L^{p/2}}\Bigl)^{1/2}\\
&\le \frac{2pC_{p/2}}{p-2}\Bigl(\sum_{k=0}^{N_{t}-1}\Delta\int_{t_{k}}^{t_{k+1}}\Bigl\|\Bigl\langle (u-\kappa(u))(f(x(\kappa(u)))-f(Y(\kappa(u)))),H(u)
\Bigl\rangle\Bigl\|^{2}_{L^{p/2}}du\Bigl)^{1/2}.
\end{align*}

To estimate $J_{41}$, we firstly show that
\begin{align}\label{D3}
&\Bigl\|\Bigl\langle (u-\kappa(u))(f(x(\kappa(u)))-f(Y(\kappa(u)))),H(u)
\Bigl\rangle\Bigl\|^{2}_{L^{p/2}}\notag\\
&\le\Bigl\|(u-\kappa(u))(f(x(\kappa(u)))-f(Y(\kappa(u))))\Bigl\|^{2}_{L^{p}}
\cdot\Bigl\|H(u)\Bigl\|^{2}_{L^{p}}\notag\\
&\le C\bar{K}_{R}^{2}\Delta^{2}\cdot CK_{R}^{2}\bar{K}^{2}_{R}\Delta\le CK_{R}^{2}\bar{K}_{R}^{4}\Delta^{3}.
\end{align}

Thus,
\begin{align*}
J_{41}\le CK_{R}\bar{K}_{R}^{2}\Delta^{2}.
\end{align*}

As for $J_{42}$, by Lemma \ref{D2}, (\ref{D3}) and H\"{o}lder inequality, we obtain
\begin{align*}
\bigl(\frac{J_{42}}{2}\bigl)^{p/2}&\le \Delta^{\frac{p}{2}-1}\mathbb{E}\Bigl(\sup_{0\le s\le t}\int_{\kappa(s)}^{s}\Bigl|\Bigl\langle (u-\kappa(u))(f(x(\kappa(u)))-f(Y(\kappa(u)))),H(u)\Bigl\rangle\Bigl|^{p/2}du\Bigl)\\
&\le \Delta^{\frac{p}{2}-1}\mathbb{E}\Bigl(\sum_{k=0}^{N_{t}-1}\int_{t_{k}}^{t_{k+1}}\Bigl|\Bigl\langle (u-\kappa(u))(f(x(\kappa(u)))-f(Y(\kappa(u)))),H(u)\Bigl\rangle\Bigl|^{p/2}du\\
&\quad+\int_{\kappa(t)}^{t}\Bigl|\Bigl\langle (u-\kappa(u))(f(x(\kappa(u)))-f(Y(\kappa(u)))),H(u)\Bigl\rangle\Bigl|^{p/2}du\Bigl)\\
&=\Delta^{\frac{p}{2}-1}\int_{0}^{t}\mathbb{E}\Bigl|\Bigl\langle (u-\kappa(u))(f(x(\kappa(u)))-f(Y(\kappa(u)))),H(u)\Bigl\rangle\Bigl|^{p/2}du\\
&\le CK_{R}^{p/2}\bar{K}_{R}\Delta^{5p/4-1}.
\end{align*}
Thus
\begin{align*}
J_{42}\le CK_{R}\bar{K}_{R}^{2}\Delta^{\frac{5}{2}-\frac{2}{p}}\le CK_{R}\bar{K}_{R}^{2}\Delta^{2}.
\end{align*}
Then
\begin{align*}
J_{4}\le CK_{R}\bar{K}_{R}^{2}\Delta^{2}.
\end{align*}
Similarly, we divide $J_{5}$ as follows:
\begin{align*}
J_{5}&\le 2\Bigl\|\sup_{0\le s\le t}\Bigl|\sum_{k=0}^{N_{s}-1}\int_{t_{k}}^{t_{k+1}}\langle x(\kappa(u\land\theta))-Y(\kappa(u\land\theta)),H(u\land\theta)\rangle du\Bigl|\Bigl\|_{L^{p/2}}\\
&\quad+2\Bigl\|\sup_{0\le s\le t}\Bigl|\int_{\kappa(s)}^{s}\langle x(\kappa(u\land\theta))-Y(\kappa(u\land\theta)),H(u\land\theta)\rangle du\Bigl|\Bigl\|_{L^{p/2}}\\
&=:J_{51}+J_{52}.
\end{align*}

For $J_{51}$, we follow the same line as $J_{41}$ to derive
\begin{align*}
J_{51}&\le\frac{2pC_{p/2}}{p-2}\Bigl(\sum_{k=0}^{N_{t}-1}\Delta\int_{t_{k}}^{t_{k+1}}
\Bigl\|x(\kappa(u\land\theta))-Y(\kappa(u\land\theta))\Bigl\|^{2}_{L^{p}}\cdot
\Bigl\|H(u\land\theta)\Bigl\rangle\Bigl\|^{2}_{L^{p}}du\Bigl)^{1/2}\\
&\le\sup_{0\le s\le t}\|e(s\land\theta)\|_{L^{p}}\cdot\frac{2pC_{p/2}}{p-2}\Bigl(\sum_{k=0}^{N_{t}-1}
\Delta\int_{t_{k}}^{t_{k+1}}\Bigl\|H(u\land\theta)\Bigl\|^{2}_{L^{p}}du\Bigl)^{1/2}\\
&\le \frac{1}{4}\sup_{0\le s\le t}\|e(s\land\theta)\|_{L^{p}}^{2}+\frac{4p^{2}C^{2}_{p/2}}{(p-2)^{2}}
\sum_{k=0}^{N_{t}-1}\Delta\int_{t_{k}}^{t_{k+1}}\Bigl\|H(u\land\theta)\Bigl\|^{2}_{L^{p}}du\\
&\le \frac{1}{4}\sup_{0\le s\le t}\|e(s\land\theta)\|_{L^{p}}^{2}+CK_{R}^{2}\bar{K}_{R}^{2}\Delta^{2}.
\end{align*}

With regard to $J_{52}$, we use the same way as $J_{42}$ to yield
\begin{align*}
\bigl(\frac{J_{52}}{2}\bigl)^{p/2}&\le\Delta^{\frac{p}{2}-1}
\int_{0}^{t}\|x(\kappa(u\land\theta))-Y(\kappa(u\land\theta))\|^{p/2}_{L^{p}}\cdot\|H(u\land\theta)\|^{p/2}_{L^{p}}du\\
&\le \sup_{0\le s\le t}\|e(s\land\theta)\|^{p/2}_{L^{p}}\cdot CK_{R}^{p/2}\bar{K}_{R}^{p/2}\Delta^{\frac{3p}{4}-1}.
\end{align*}
Thus
\begin{align*}
J_{52}&\le 2\sup_{0\le s\le t}\|e(s\land\theta)\|_{L^{p}}\cdot C\bar{K}_{R}K_{R}\Delta^{(3p-4)/2p}\\
&\le\frac{1}{4}\sup_{0\le s\le t}\|e(s\land\theta)\|_{L^{p}}^{2}+CK_{R}^{2}\bar{K}_{R}^{2}\Delta^{(3p-4)/p}.
\end{align*}

Note that $(3p-4)/p\ge 2$ for $p\ge 4$, we have
\begin{align*}
J_{5}\le \frac{1}{2}\sup_{0\le s\le t}\|e(s\land\theta)\|_{L^{p}}^{2}+CK_{R}^{2}\bar{K}_{R}^{2}\Delta^{2}.
\end{align*}

Therefore, we can get the following estimate of $J$
\begin{align*}
J\le mp^{2}\bar{K}_{R}^{2}\int_{0}^{t}\sup_{0\le r\le u}\|e(u\land\theta)\|_{L^{p}}^{2}du+\frac{1}{2}\sup_{0\le s\le t}\|e(s\land\theta)\|_{L^{p}}^{2}+CK_{R}^{2}\bar{K}_{R}^{3}\Delta^{2}.
\end{align*}
Hence, we have
\begin{align*}
&\frac{3}{4}\Bigl\|\sup_{0\le s\le t}|e(s\land\theta)|\Bigl\|^{2}_{L^{p}}\le C\bar{K}_{R}^{2}\int_{0}^{t}\|e(s\land\theta)\|^{2}_{L^{p}}ds+CK_{R}^{2}\bar{K}_{R}^{4}\Delta^{2}\\
&+mp^{2}\bar{K}_{R}^{2}\int_{0}^{t}\sup_{0\le r\le u}\|e(u\land\theta)\|_{L^{p}}^{2}du+\frac{1}{2}\sup_{0\le s\le t}\|e(s\land\theta)\|_{L^{p}}^{2}+CK_{R}^{2}\bar{K}_{R}^{3}\Delta^{2}.
\end{align*}
Since $K_R\le \sqrt{\bar{K}_R}$ by Lemma \ref{l0}, Gronwall's inequality yields
\begin{align*}
\Bigl\|\sup_{0\le s\le t}|e(s\land\theta)|\Bigl\|_{L^{p}}\le C K_{R}\bar{K}_{R}^{2}\exp(C\bar{K}^{2}_{R})\cdot\Delta\le C \bar{K}_{R}^{5/2}\exp(C\bar{K}^{2}_{R})\cdot\Delta.
\end{align*}
Without loss of generality, suppose $\Delta^{\ast}$ is sufficiently small such that
\begin{align*}
h(\Delta^{\ast})\ge \bar{K}^{-1}\left(\bar{K}_{R}\exp(2C\bar{K}^{2}_{R}/5)\right)\ge R.
\end{align*}
Here, $\bar{K}^{-1}$ is the inverse of $\bar{K}_{R}$ as a function of $R.$

Then we can get
\begin{align*}
\Bigl\|\sup_{0\le s\le t}|e(s\land\theta)|\Bigl\|_{L^{p}}\le C\bar{K}^{5/2}_{h(\Delta)}\Delta.
\end{align*}
We complete the proof.
\end{proof}
\end{Theorem}
\begin{Theorem}\label{N95}
Let Assumptions $\ref{A3}$ and $\ref{A2}$ hold for some $p>4$. If there exists a $\Delta^{\ast}$ sufficiently small such that $(\ref{B8})$ and
\begin{equation}
h(\Delta)\ge (\Delta^{q}(\bar{K}_{h(\Delta)})^{5q/2}))^{-\frac{1}{p-q}}
\end{equation}
hold for all $\Delta\in(0,\Delta^{\ast}]$, then for any fixed $T>0$ and $q<p$,
\begin{equation}
\|x(T)-Y(T)\|_{L^{q}}\le C\bar{K}^{5/2}_{h(\Delta)}\Delta
\end{equation}
holds, where C is a positive constant independent of $\Delta$.
\end{Theorem}
For $q\ge4,$ the proof is similar to Theorem in \cite{AC}. And Jensen's inequality implies the required result for $q<4$, so we omit it here.

\section{ Exponential stability of modified truncated Milstein method}

To investigate the exponential stability of the modified truncated EM method, the following assumption is required throughout this Section.
\begin{Assumption}
Assume that $f(0)=0, g(0)=0$, and the coefficients satisfy the local Lipschitz condition, i.e., for each $R>0$ there is $K_{R}>0$ (depending on $R$) such that
\begin{equation}\label{E1}
|f(x)-f(y)|\vee|g_{j}(x)-g_{j}(y)|\vee|L^{j_{1}}g_{j_{2}}(x)-L^{j_{1}}g_{j_{2}}(y)|\le K_{R}|x-y|
\end{equation}
for all $|x|\vee|y|\le R$ and $j,j_{1},j_{2}=1,2,...,m$.
\end{Assumption}
\begin{Assumption}\label{E2}
There exists a pair of positive constants $\lambda$ and $p$ such that
\begin{equation}\label{E21}
\langle x,f(x)\rangle + \frac{p-1}{2}\sum_{j=1}^{m}|g_{j}(x)|^{2}\le -\lambda|x|^{2}.
\end{equation}
\end{Assumption}
In addition, we choose $h(\Delta)$ such that $$\lim_{\Delta\rightarrow 0}K^{6}_{h(\Delta)}\Delta=0.$$

As interpreted in \cite{AC}, such $h(\Delta)$ always exists.

For the modified functions $\tilde{f}$ and $\tilde{g}$, we have
\begin{Lemma}
Suppose $(\ref{E21})$ for holds for $p> 0$. Then we have
\begin{equation}
\langle x,\tilde{f}(x)\rangle+\frac{p-1}{2}\sum_{j=1}^{m}|\tilde{g}_{j}(x)|^{2}\le -\lambda|x|^{2}.
\end{equation}
\end{Lemma}
For the proof, see Lemma 2.2 in \cite{AC}.
\begin{Theorem}\label{E3}
Suppose Assumptions $\ref{E1}$, $\ref{E2}$ hold. Then for any sufficiently small $\varepsilon\in (0,1)$, there exists $\Delta^{\ast}>0$ such that for any $\Delta\in (0,\Delta^{\ast})$ the discrete approximation satisfies
\begin{equation}\label{E31}
\lim_{k\rightarrow \infty}\sup\frac{\log\mathbb{E}(|Y_{k}|^{p})}{k\Delta}\le -p(\lambda-\varepsilon).
\end{equation}
For the continuous-times approximation, we also have
\begin{equation}\label{E32}
\lim_{t\rightarrow \infty}\sup\frac{\log\mathbb{E}(|Y(t)|^{p})}{t}\le -p(\lambda-\varepsilon).
\end{equation}
\end{Theorem}
\begin{proof}
For any $k\ge 0$, we have
\begin{align*}
|Y_{k+1}|^{2}&=|Y_{k}|^{2}+2\langle Y_{k},\tilde{f}(Y_{k})\Delta+\sum_{j=1}^{m}\tilde{g}_{j}(Y_{k})\Delta B^{j}_{k}+A_{k}\rangle+|\tilde{f}(Y_{k})\Delta+\sum_{j=1}^{m}\tilde{g}_{j}(Y_{k})\Delta B^{j}_{k}+A_{k}|^{2}\\
&=|Y_{k}|^{2}(1+\xi_{k})
\end{align*}
where
\begin{align*}
A_{k}=\frac{1}{2}\sum_{j_{1}=1}^{m}\sum_{j_{2}=1}^{m}L^{j_{1}}\tilde{g}_{j_{2}}(Y_{k})\Delta B^{j_{2}}_{k}\Delta B^{j_{1}}_{k}-\frac{1}{2}\sum_{j=1}^{m}L^{j}\tilde{g}_{j}(Y_{k})\Delta
\end{align*}
and
\begin{align*}
\xi_{k}=\frac{1}{|Y_{k}|^{2}}(2\langle Y_{k},\tilde{f}(Y_{k})\Delta+\sum_{j=1}^{m}\tilde{g}_{j}(Y_{k})\Delta B^{j}_{k}+A_{k}\rangle+|\tilde{f}(Y_{k})\Delta+\sum_{j=1}^{m}\tilde{g}_{j}(Y_{k})\Delta B^{j}_{k}+A_{k}|^{2})
\end{align*}
if $Y_{k}\ne 0$, otherwise, it is set to $-1$. So we only need to suppose that $Y_{k}\ne 0$.\\
Then we have
\begin{align*}
|Y_{k+1}|^{p}=|Y_{k}|^{p}(1+\xi_{k})^{\frac{p}{2}}.
\end{align*}
By Taylor's expansion,
\begin{align*}
(1+x)^{\frac{p}{2}}=1+\frac{p}{2}x+\frac{p(p-2)}{8}x^{2}+\frac{p(p-2)(p-4)}{2^{3}\times3!}(1+\theta x)^{\frac{p}{2}-3}x^{3}
\end{align*}
where $\theta\in (0,1)$.
Then we derive that
\begin{align*}
\mathbb{E}(|Y_{k+1}|^{p}|\mathscr{F}_{k\Delta})=|Y_{k}|^{p}\mathbb{E}\bigl(1+\frac{p}{2}\xi_{k}+\frac{p(p-2)}{8}\xi^{2}_{k}+\frac{p(p-2)(p-4)}{2^{3}\times3!}(1+\theta \xi_{k})^{\frac{p}{2}-3}\xi_{k}^{3}\bigl|\mathscr{F}_{k\Delta}\bigl).
\end{align*}

Notice that
\begin{align*}
\mathbb{E}(\xi_{k}|\mathscr{F}_{k\Delta})&=\frac{1}{|Y_{k}|^{2}}\mathbb{E}\Bigl(2\bigl\langle Y_{k},\tilde{f}(Y_{k})\Delta+\sum_{j=1}^{m}\tilde{g}_{j}(Y_{k})\Delta B^{j}_{k}+A_{k}\bigl\rangle+\\&\qquad\qquad\bigl|\tilde{f}(Y_{k})\Delta+\sum_{j=1}^{m}\tilde{g}_{j}(Y_{k})\Delta B^{j}_{k}+A_{k}\bigl|^{2}\Bigl|\mathscr{F}_{k\Delta}\Bigl)\\
&=\frac{1}{|Y_{k}|^{2}}\Bigl(\bigl(2\bigl\langle Y_{k},\tilde{f}(Y_{k})\bigl\rangle+\sum_{j=1}^{m}|\tilde{g}_{j}(Y_{k})|^{2}\bigl)\Delta
+|\tilde{f}(Y_{k})|^{2}\Delta^{2}\Bigl)+\frac{1}{|Y_{k}|^{2}}
\mathbb{E}(A_{k}^{2}|\mathscr{F}_{k\Delta}).
\end{align*}

We have used the fact that
$$\mathbb{E}(\Delta B_{k}^{i})^{2n+1}|\mathscr{F}_{k\Delta})=0, \mathbb{E}((\Delta B_{k}^{i})^{2n}|\mathscr{F}_{k\Delta})\le C\Delta^{n}$$ and
$\mathbb{E}(\Delta B_{k}^{i}\Delta B_{k}^{j}|\mathscr{F}_{k\Delta})=\delta_{ij}\Delta$. Applying Lemma 3.1, we obtain
\begin{align*}
&|L^{j_{1}}\tilde{g}_{j_{2}}(Y_{k})|
=|L^{j_{1}}\tilde{g}_{j_{2}}(Y_{k})-L^{j_{1}}\tilde{g}_{j_{2}}(0)|\le 4K_{h(\Delta)}|Y_{k}|.
\end{align*}

Meanwhile,
\begin{align*}
\frac{1}{|Y_{k}|^{2}}\mathbb{E}(A_{k}^{2}|\mathscr{F}_{k\Delta})&=\frac{1}{|Y_{k}|^{2}}\cdot\mathbb{E}\Bigl(\Bigl|\frac{1}{2}\sum_{j_{1}=1}^{m}\sum_{j_{2}=1}^{m}L^{j_{1}}\tilde{g}_{j_{2}}(Y_{k})(\Delta B_{k}^{j_{2}}\Delta B_{k}^{j_{1}}-\delta_{j_{1}j_{2}}\Delta)\Bigl|^{2}\Bigl|\mathscr{F}_{k\Delta}\Bigl)\\
&\le \frac{1}{|Y_{k}|^{2}}\cdot\frac{1}{4}m^{2}\sum_{j_{1}=1}^{m}\sum_{j_{2}=1}^{m}\Bigl|L^{j_{1}}\tilde{g}_{j_{2}}(Y_{k})\Bigl|^{2}\mathbb{E}\Bigl((\Delta B_{k}^{j_{2}}\Delta B_{k}^{j_{1}}-\delta_{j_{1}j_{2}}\Delta)^{2}\Bigl|\mathscr{F}_{k\Delta}\Bigl)\\
&\le CK_{h(\Delta)}^{2}\Delta^{2}=o(\Delta)
\end{align*}
and
\begin{align*}
|\tilde{f}(Y_{k})|^{2}\Delta^{2}=|\tilde{f}(Y_{k})-\tilde{f}(0)|^{2}\Delta^{2}\le 16K_{h(\Delta)}^{2}\Delta^2 |Y_{k}|^{2}=o(\Delta)|Y_{k}|^{2}.
\end{align*}

Therefore, we have
\begin{align*}
\mathbb{E}(\xi_{k}|\mathscr{F}_{k\Delta})&\le \frac{1}{|Y_{k}|^{2}}\Bigl(\bigl(2\bigl\langle Y_{k},\tilde{f}(Y_{k})\bigl\rangle
+\sum_{j=1}^{m}|\tilde{g}_{j}(Y_{k})|^{2}\bigl)\Delta\Bigl)+o(\Delta).
\end{align*}

Secondly, we use the same way as above to yield
\begin{align*}
\mathbb{E}(\xi_{k}^{2}|\mathscr{F}_{k\Delta})&=\frac{1}{|Y_{k}|^{4}}\mathbb{E}\Bigl(\bigl(2\bigl\langle Y_{k},\sum_{j=1}^{m}\tilde{g}_{j}(Y_{k})\Delta B^{j}_{k}\bigl\rangle+B\bigl)^{2}\Bigl|\mathscr{F}_{k\Delta}\Bigl)\\
&=\frac{1}{|Y_{k}|^{4}}\mathbb{E}\Bigl(4\bigl\langle Y_{k},\sum_{j=1}^{m}\tilde{g}_{j}(Y_{k})\Delta B^{j}_{k}\bigl\rangle^{2}+B^{2}+4B\bigl\langle Y_{k},\sum_{j=1}^{m}\tilde{g}_{j}(Y_{k})\Delta B^{j}_{k}\bigl\rangle\Bigl|\mathscr{F}_{k\Delta}\Bigl)
\end{align*}
where
\begin{align*}
B:=2\langle Y_{k},\tilde{f}(Y_{k})\Delta+A_{k}\rangle+
|\tilde{f}(Y_{k})\Delta+\sum_{j=1}^{m}\tilde{g}_{j}(Y_{k})\Delta B_{k}^{j}+A_{k}|^{2}.
\end{align*}
To get the estimate of $\mathbb{E}(\xi_{k}^{2}|\mathscr{F}_{k\Delta})$, we need show that
\begin{align*}
\frac{1}{|Y_{k}|^{4}}\mathbb{E}(B^{2}|\mathscr{F}_{k\Delta})&\le\frac{3}{|Y_{k}|^{4}}\cdot\bigl(4|Y_{k}|^{2}\cdot|\tilde{f}(Y_{k})|^{2}\Delta^{2}\bigl)+\frac{3}{|Y_{k}|^{4}}\cdot|Y_{k}|^{2}\mathbb{E}(|A_{k}|^{2}|\mathscr{F}_{k\Delta})\\
&\quad+\frac{3}{|Y_{k}|^{4}}\mathbb{E}\Bigl(|\tilde{f}(Y_{k})\Delta+\sum_{j=1}^{m}\tilde{g}_{j}(Y_{k})\Delta B_{k}^{j}+A_{k}|^{4}\Bigl|\mathscr{F}_{k\Delta}\Bigl)\\
&\le \frac{3}{|Y_{k}|^{4}}\cdot4|Y_{k}|^{2}\cdot 16|K_{h(\Delta)}|^{2}|Y_{k}|^{2}\Delta^{2}+\frac{3}{|Y_{k}|^{4}}\cdot|Y_{k}|^{2}\cdot|Y_{k}|^{2}o(\Delta)\\
&\quad+\frac{3^{4}}{|Y_{k}|^{4}}\bigl(|\tilde{f}(Y_{k})|^{4}\Delta^{4}
+C\sum^{m}_{j=1}|\tilde{g}_{j}(Y_{k})|^{4}\Delta^{2}
+\mathbb{E}(|A_{k}|^{4}|\mathscr{F}_{k\Delta})\bigl).
\end{align*}

Since
\begin{align*}
\mathbb{E}(|A_{k}|^{4}|\mathscr{F}_{k\Delta})&=\mathbb{E}\Bigl(\Bigl|\frac{1}{2}\sum_{j_{1}=1}^{m}\sum_{j_{2}=1}^{m}L^{j_{1}}\tilde{g}_{j_{2}}(Y_{k})(\Delta B_{k}^{j_{2}}\Delta B_{k}^{j_{1}}-\delta_{j_{1}j_{2}}\Delta)\Bigl|^{4}\Bigl|\mathscr{F}_{k\Delta}\Bigl)\\
&\le C\sum_{j_{1}=1}^{m}\sum_{j_{2}=1}^{m}\Bigl|L^{j_{1}}\tilde{g}_{j_{2}}(Y_{k})\Bigl|^{4}\mathbb{E}\Bigl((\Delta B_{k}^{j_{2}}\Delta B_{k}^{j_{1}}-\delta_{j_{1}j_{2}}\Delta)^{4}\Bigl|\mathscr{F}_{k\Delta}\Bigl)\\
&\le CK_{h(\Delta)}^{4}|Y_{k}|^{4}\cdot\Delta^{4}\\&
\le C|Y_{k}|^{4}o(\Delta),
\end{align*}
then we have
$$\frac{1}{|Y_{k}|^{4}}\mathbb{E}(B^{2}|\mathscr{F}_{k\Delta})\le o(\Delta).$$

Note that
$$\aligned &B\bigl\langle Y_{k},\sum_{j=1}^{m}\tilde{g}_{j}(Y_{k})\Delta B^{j}_{k}\bigl\rangle\\&
=2\langle Y_{k},\tilde{f}(Y_{k})\Delta\rangle\cdot\Bigl\langle Y_{k},\sum_{j=1}^{m}\tilde{g}_{j}(Y_{k})\Delta B^{j}_{k}\Bigl\rangle
+2\langle Y_{k},A_{k}\rangle\cdot\Bigl\langle Y_{k},\sum_{j=1}^{m}\tilde{g}_{j}(Y_{k})\Delta B^{j}_{k}\Bigl\rangle\\&
\quad+|\tilde{f}(Y_{k})|^{2}\Delta^{2}\Bigl\langle Y_{k},\sum_{j=1}^{m}\tilde{g}_{j}(Y_{k})\Delta B^{j}_{k}\Bigl\rangle+
\Bigl\langle Y_{k},|\sum_{j=1}^{m}\tilde{g}_j(Y_{k})\Delta B_{k}^{j}|^2\sum_{j=1}^{m}\tilde{g}_j(Y_{k})\Delta B_{k}^{j}\Bigl\rangle\\&
\quad+\Bigl\langle Y_{k},A_{k}^{2}\cdot\sum_{j=1}^{m}\tilde{g}_{j}(Y_{k})\Delta B^{j}_{k}\Bigl\rangle+\Bigl\langle \tilde{f}(Y_{k})\Delta,\sum_{j=1}^{m}\tilde{g}_{j}(Y_{k})\Delta B^{j}_{k}\Bigl\rangle\cdot\Bigl\langle Y_{k},\sum_{j=1}^{m}\tilde{g}_{j}(Y_{k})\Delta B^{j}_{k}\Bigl\rangle\\
&\quad+\Bigl\langle \tilde{f}(Y_{k})\Delta,A_{k}\Bigl\rangle\cdot\Bigl\langle Y_{k},\sum_{j=1}^{m}\tilde{g}_{j}(Y_{k})\Delta B^{j}_{k}\Bigl\rangle
+\Bigl\langle \sum_{j=1}^{m}\tilde{g}_{j}(Y_{k})\Delta B^{j}_{k},A_{k}\Bigl\rangle\cdot\Bigl\langle Y_{k},\sum_{j=1}^{m}\tilde{g}_{j}(Y_{k})\Delta B^{j}_{k}\Bigl\rangle\quad\endaligned$$

It is easy to obtain that the conditional expectations of the first five terms and the seventh term are all zero.

Then we have
\begin{align*}
&\quad\frac{1}{|Y_{k}|^{4}}\mathbb{E}\Bigl(B\bigl\langle Y_{k},\sum_{j=1}^{m}\tilde{g}_{j}(Y_{k})\Delta B^{j}_{k}\bigl\rangle\Bigl|\mathscr{F}_{k\Delta}\Bigl)\\
&=\frac{2}{|Y_{k}|^{4}}\mathbb{E}\Bigl(\Bigl\langle \tilde{f}(Y_{k})\Delta,\sum_{j=1}^{m}\tilde{g}_{j}(Y_{k})\Delta B^{j}_{k}\Bigl\rangle\cdot\Bigl\langle Y_{k},\sum_{j=1}^{m}\tilde{g}_{j}(Y_{k})\Delta B^{j}_{k}\Bigl\rangle\Bigl|\mathscr{F}_{k\Delta}\Bigl)\\
&\quad+\frac{2}{|Y_{k}|^{4}}\mathbb{E}\Bigl(\Bigl\langle \sum_{j=1}^{m}\tilde{g}_{j}(Y_{k})\Delta B^{j}_{k},A_{k}\Bigl\rangle\cdot\Bigl\langle Y_{k},\sum_{j=1}^{m}\tilde{g}_{j}(Y_{k})\Delta B^{j}_{k}\Bigl\rangle\Bigl|\mathscr{F}_{k\Delta}\Bigl)\\
&\le \frac{2}{|Y_{k}|^{4}}\cdot |\tilde{f}(Y_{k})|\Delta\cdot|Y_{k}|\cdot(\sum_{j=1}^{m}|\tilde{g}_{j}(Y_{k})|^{2}\Delta)\\
&\quad+\frac{2}{|Y_{k}|^{4}}\cdot|Y_{k}|\sum_{j_1,j_2,j_3,j_4=1}^{m}
|L^{j_1}\tilde{g}_{j_{2}}(Y_{k})|\cdot|\tilde{g}_{j_3}(Y_{k})|\cdot|\tilde{g}_{j_4}(Y_{k})|
\\&\qquad\qquad\qquad\qquad\cdot\mathbb{E}\Bigl(\Bigl|\Delta B^{j_3}_{k}\Delta B^{j_4}_{k}(\Delta B_{k}^{j_{1}}\Delta B_{k}^{j_{2}}-\delta_{j_{1}j_{2}}\Delta)\Bigl|\Bigl)\\
&\le CK_{h(\Delta)}^{3}\Delta^{2} +CK_{h(\Delta)}^{3}\Delta^{2}\le o(\Delta).
\end{align*}

Therefore,
\begin{align*}
\mathbb{E}(\xi_{k}^{2}|\mathscr{F}_{k\Delta})\le \frac{4}{|Y_{k}|^{2}}\sum_{j=1}^{m}|\tilde{g}_{j}(Y_{k})|^{2}\Delta+o(\Delta).
\end{align*}

For $p\ge 6$, we have
\begin{align*}
\mathbb{E}(|\xi_{k}|^{\frac{p}{2}}|\mathscr{F}_{k\Delta})&\le \frac{1}{|Y_{k}|^{p}}C_{p}\Bigl(|Y_{k}|^{\frac{p}{2}}
|\tilde{f}(Y_{k})|^{\frac{p}{2}}\Delta^{\frac{p}{2}}
+|Y_{k}|^{\frac{p}{2}}|\sum_{j=1}^{m}\tilde{g}_{j}(Y_{k})|^{\frac{p}{2}}\cdot\mathbb{E}(|\Delta B_{k}|^{\frac{p}{2}}|\mathscr{F}_{k\Delta})\\
&\quad+|Y_{k}|^{\frac{p}{2}}|\mathbb{E}(|A_{k}|^{\frac{p}{2}}|\mathscr{F}_{k\Delta})
+|\tilde{f}(Y_{k})|^{p}\Delta^{p}
\\&\quad+\Bigl|\sum_{j=1}^{m}\tilde{g}_{j}(Y_{k})\Bigl|^{p}\cdot\mathbb{E}(|\Delta B_{k}|^{p}|\mathscr{F}_{k\Delta})+\mathbb{E}(|A_{k}|^{p}|\mathscr{F}_{k\Delta})\Bigl)\\
&\le CK_{h(\Delta)}^{\frac{p}{2}}\Delta^{\frac{p}{2}}+CK_{h(\Delta)}^{\frac{p}{2}}\Delta^{\frac{p}{4}}
+\frac{C}{|Y_{k}|^{\frac{p}{2}}}\mathbb{E}(|A_{k}|^{\frac{p}{2}}|\mathscr{F}_{k\Delta})
\\&\quad+K_{h(\Delta)}^{p}\Delta^{p}+K_{h(\Delta)}^{p}\Delta^{\frac{p}{2}}
+\frac{C}{|Y_{k}|^{p}}\mathbb{E}(|A_{k}|^{p}|\mathscr{F}_{k\Delta}).
\end{align*}

Note that
\begin{align*}
\mathbb{E}(|A_{k}|^{i}|\mathscr{F}_{k\Delta})&=\mathbb{E}\Bigl(\Bigl|\frac{1}{2}\sum_{j_{1}=1}^{m}\sum_{j_{2}=1}^{m}L^{j_{1}}\tilde{g}_{j_{2}}(Y_{k})(\Delta B_{k}^{j_{2}}\Delta B_{k}^{j_{1}}-\delta_{j_{1}j_{2}}\Delta)\Bigl|^{i}\Bigl|\mathscr{F}_{k\Delta}\Bigl)\\
&\le CK_{h(\Delta)}^{i}\Delta^{i}|Y_{k}|^{i}.
\end{align*}

Thus we obtain
\begin{align*}
\mathbb{E}(|\xi_{k}|^{\frac{p}{2}}|\mathscr{F}_{k\Delta})\le CK_{h(\Delta)}^{\frac{p}{2}}\Delta^{\frac{p}{4}}=o(\Delta).
\end{align*}
Indeed,
\begin{align*}
\lim_{\Delta\rightarrow 0}\frac{K_{h(\Delta)}^{\frac{p}{2}}\Delta^{\frac{p}{4}}}{\Delta}=(K_{h(\Delta)}\Delta^{\frac{1}{6}}\cdot\Delta^{\frac{1}{3}-\frac{2}{p}})^{\frac{p}{2}}=0
\end{align*}
where $\frac{1}{3}-\frac{2}{p}\ge 0$ for $p\ge 6$.

Moreover, since
$$\mathbb{E}((1+\theta\xi_k)^{\frac{p}{2}-3}\xi_k^3|\mathscr{F}_{k\Delta})\le C\mathbb{E}((|\xi_k|^3+|\xi_k|^\frac{p}{2})|\mathscr{F}_{k\Delta})\le C(K_{h(\Delta)}^{3}\Delta^{\frac{3}{2}}+K_{h(\Delta)}^{\frac{p}{2}}\Delta^{\frac{p}{4}})=o(\Delta),$$
We have
\begin{align*}
\mathbb{E}(|Y_{k+1}|^{p}|\mathscr{F}_{k\Delta})&\le |Y_{k}|^{p}\Bigl(1+\frac{p}{2}\cdot\frac{1}{|Y_{k}|^{2}}\bigl(2\bigl\langle Y_{k},\tilde{f}(Y_{k})\bigl\rangle+\sum_{j=1}^{m}|\tilde{g}_{j}(Y_{k})|^{2}\bigl)\Delta\\
& \quad+\frac{p(p-2)}{8}\cdot \frac{4}{|Y_{k}|^{2}}\sum_{j=1}^{m}|\tilde{g}_{j}(Y_{k})|^{2}\Delta+o(\Delta) \Bigl)                                          \\
&=|Y_{k}|^{p}\Bigl(1+p\cdot\frac{1}{|Y_{k}|^{2}}\bigl(2\bigl\langle Y_{k},\tilde{f}(Y_{k})\bigl\rangle+\frac{(p-1)}{2}\sum_{j=1}^{m}|\tilde{g}_{j}(Y_{k})|^{2}\bigl)\Delta+o(\Delta)\Bigl)\\
&\le|Y_{k}|^{p}(1-p\lambda\Delta+o(\Delta)).
\end{align*}

Now for any given $\varepsilon\in(0,\lambda)$ sufficiently small, we can choose $\Delta^{\ast}\in(0,1)$ sufficiently small such that for all $\Delta\in(0,\Delta^{\ast})$ $\frac{o(\Delta)}{\Delta}\le p\varepsilon$.

Thus,
\begin{align*}
\mathbb{E}(|Y_{k+1}|^{p}|\mathscr{F}_{k\Delta})&\le|Y_{k}|^{p}(1-p(\lambda-\varepsilon)\Delta).
\end{align*}
Taking expectation on both sides yields
\begin{align*}
\mathbb{E}(|Y_{k+1}|^{p})\le \mathbb{E}(|Y_{k}|^{p})(1-p(\lambda-\varepsilon)\Delta).
\end{align*}
Thus, for any $k\ge 1$,
\begin{align*}
 \mathbb{E}(|Y_{k}|^{p})\le |x(0)|^{p}(1-p(\lambda-\varepsilon)\Delta)^{k}\le|x(0)|^{p}e^{-kp(\lambda-\varepsilon)\Delta}.
\end{align*}
So pth moment exponential stability of $Y_k$, or (\ref{E31}), hold.

In the end, we give the proof of (\ref{E32}). Given any fixed $\Delta>0$ and suppose $k\Delta\le t<(k+1)\Delta$. Then
\begin{align*}
Y(t)-\bar{Y}(t)&=\tilde{f}(Y_{k})(t-k\Delta)+\sum_{j=1}^{m}\tilde{g}_{j}(Y_{k})(B^{j}(t)-B^{j}(k\Delta))\\
&+\frac{1}{2}\sum_{j_{1}=1}^{m}\sum_{j_{2}=1}^{m}L^{j_{1}}
\tilde{g}_{j_{2}}(B^{j_{2}}(t)B^{j_{1}}(t)-B^{j_{2}}(k\Delta)B^{j_{1}}(k\Delta)).
\end{align*}

Therefore, for $\Delta$ small enough, there exists $C>0$  such that
\begin{align*}
\mathbb{E}(|Y(t)|^{p})&\le 2^{p-1}\bigl(\mathbb{E}(|Y(t)-Y_{k}|^{p})+\mathbb{E}(|Y_{k}|^{p})\bigl)\\
&\le C\mathbb{E}(|Y_{k}|^{p}).
\end{align*}
It follows that
\begin{align*}
\limsup_{t\rightarrow 0}\frac{\log\mathbb{E}(|Y(t)|^{p})}{t}&\le\limsup_{k\rightarrow \infty}\sup_{k\Delta\le t<(k+1)\Delta}\frac{\log\mathbb{E}(|Y(t)|^{p})}{t}\\
&\le\limsup_{k\rightarrow \infty}\sup_{k\Delta\le t<(k+1)\Delta}\frac{\log C+\log\mathbb{E}(|Y_{k}|^{p})}{t}\\
&\le\limsup_{k\rightarrow \infty}\frac{\log C+p\log|x_{0}|-kp(\lambda-\varepsilon)\Delta}{(k+1)\Delta}\\
&=-p(\lambda-\varepsilon).
\end{align*}
We complete the proof.
\end{proof}

\section{Numerical examples}

Now let us present some examples to interpret our Theorem \ref{N95} and Theorem \ref{E3}.\\

\textbf{Example 1} Consider the following  stochastic differential equation
\begin{equation}
\aligned dx(t)=\begin{pmatrix}
x_{1}(t)-2x_{1}(t)e^{|x(t)|}-x_{2}(t)e^{|x(t)|}\\
x_{2}(t)+x_{1}(t)e^{|x(t)|}-2x_{2}(t)e^{|x(t)|}\\
\end{pmatrix}dt+\begin{pmatrix}
x_{1}(t)e^{\frac{|x(t)|}{2}} \\
x_{2}(t)e^{\frac{|x(t)|}{2}} \\
\end{pmatrix}dB_{t}\endaligned
\end{equation}
with initial value $x_0$. It is obvious that neither $f(x)$ nor $g(x)$ is polynomial growing (although both are local Lipschitz continuous). However, we can show that conditions (\ref{B17}) holds for $p=5$. That is to say
\begin{equation}
x^{T}f(x)+\frac{p-1}{2}|g(x)|^{2}=|x|^{2}\le 1+|x|^{2}.
\end{equation}
Moreover, in this case, Assumption \ref{A3} holds for $K_{R}=3Re^{R}$. Then we can get $\bar{K}_{R}=2RK^{2}_{R}=18R^3e^{2R}$, where $R$ is big enough.

Then for any $0<\varepsilon<1$, we can choose $l(x)=\frac{1}{18^{5/2}x^{15/2+\varepsilon}e^{5x}}$ for $x>0$. It is clear that $l$ is strictly decreasing in $(0,\infty)$. Let $h$ be the inverse function of $l$. Then $h$ is strictly decreasing in $(0,\Delta^{\ast})$ for sufficiently small $\Delta^*$ and $h(\Delta)\rightarrow\infty$ as $\Delta\rightarrow 0$. Therefore, we have $\bar{K}_{h(\Delta)}^{5/2}\Delta=\bar{K}_{h(\Delta)}^{5/2}l(h(\Delta))
=\frac{18^{5/2}h(\Delta)^{15/2}e^{5h(\Delta)}}{18^{5/2}h(\Delta)^{15/2+\varepsilon}e^{5h(\Delta)}}=h(\Delta)^{-\varepsilon}\to 0$ as $\Delta\to0$. At the same time
\begin{align*}
(\Delta^{q}\bar{K}_{h(\Delta)}^{5q/2})^{\frac{-1}{p-q}}=(h(\Delta)^{-\varepsilon})^{\frac{-q}{p-q}}=(h(\Delta))^{\frac{\varepsilon q}{p-q}}\le h(\Delta)
\end{align*}
for $q\le p/(1+\varepsilon)$.

Then by Theorem \ref{N95}, for any $T>0$ and sufficient small $\Delta$, we have
\begin{equation}
\|x(T)-Y(T)\|_{L^q}\le C(h(\Delta))^{-\varepsilon}\rightarrow 0
\end{equation}
for all $q\le 5/(1+\varepsilon)$.

Although the convergence rate is extremely slow, it at least indicates that our convergence result holds for some cases when the coefficients are exponentially growing.

\textbf{Example 2} Consider the following SDE
\begin{equation}\label{N96}
\aligned dx(t)=\begin{pmatrix}
1 - 3x_{1}^3(t) + x_{2}(t)\\
x_{1}(t)\\
\end{pmatrix}dt+\begin{pmatrix}
x_{1}^2(t) & 0 \\
0 & x_{2}(t) \\
\end{pmatrix}dB_{t}\endaligned
\end{equation}
with the initial value $x_{0} = (1,1)^{T}$. It is easy to know that $K_{R}=9R^{2}$ and $\bar{K}_{R}=81R^{5}$ for Assumption \ref{A3}. And \ref{A2} holds for $p=7$ since
\begin{align*}
x^{T}f(x)+\frac{p-1}{2}|g(x)|^2=x_1-3x_1^4+x_1x_2 + x_1x_2 + 3x_1^4 + 3x_2^2 \le 4(1 + x_1^2 + x_2^2) = 4(1+|x|^2).
\end{align*}
Besides, the diffusion coefficient $g$ satisfies the commutativity condition.
We choose $h(\Delta)=\Delta^{-\frac{2\varepsilon}{25}}$ where $\varepsilon>0$ small enough. Then $\bar{K}_{h(\Delta)}^{5/2}\Delta=3^{10}\Delta^{1-\varepsilon}\to0$. Moreover,
\begin{align*}
(\Delta \bar{K}^{5/2}_{h(\Delta)})^{\frac{-q}{p-q}}\le(1/\Delta)^{(1-\varepsilon)q/(p-q)}\le \Delta^{-2\varepsilon/25}=h(\Delta)
\end{align*}
holds for $q\le \frac{14\varepsilon }{25-23\varepsilon}$. Therefore, by Theorem \ref{N95}, we obtain
\begin{align*}
\|x(T)-Y(T)\|_{L^{q}}\le C\Delta^{1-\varepsilon}.
\end{align*}
That is, the convergence rate is close to 1.

For the computer simulations, we regard the numerical solution with the step size of $2^{-15}$ as the true solution and plot the strong errors of the modified truncated Milstein method with step sizes $2^{-11},2^{-10},2^{-9}$ and $2^{-8}$, respectively.

\begin{figure}[htbp]
  \centering
  \includegraphics[width=0.3\textwidth]{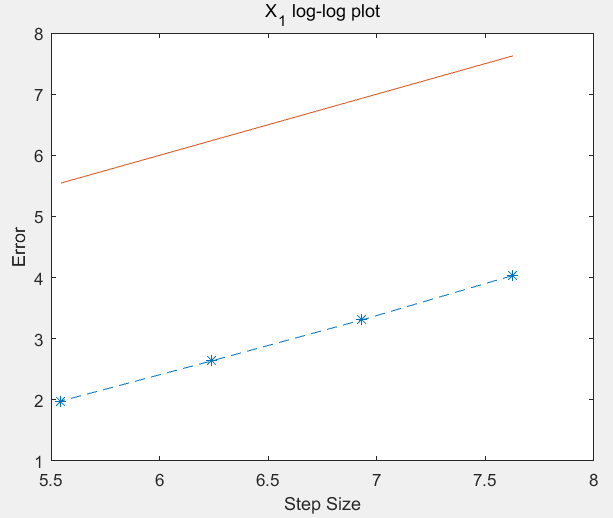}
  \caption{The strong convergence order at the terminal time $T = 1$. The red dashed line is the reference line with the slope of 1.}\label{fig:digit}
\end{figure}

\begin{figure}[htbp]
  \centering
  \includegraphics[width=0.3\textwidth]{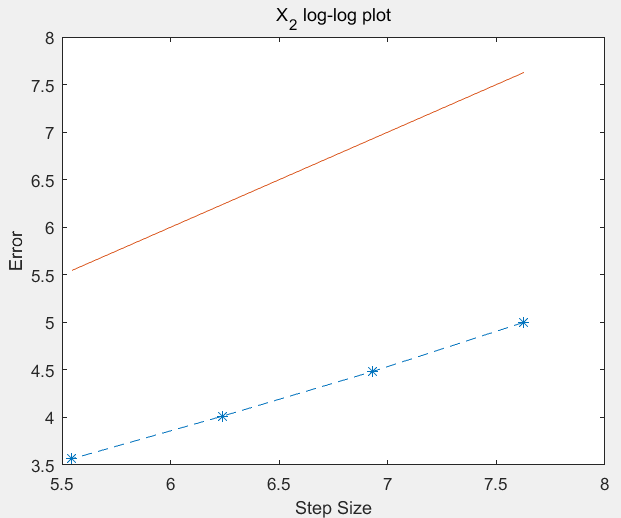}
  \caption{The strong convergence order at the terminal time $T = 1$. The red dashed line is the reference line with the slope of 1.}\label{fig:digit}
\end{figure}

Figures 1 and 2 indicate that our theoretical result holds for the polynomial growing case, and the convergence rate of the two components can be arbitrarily close to 1.

\textbf{Example 3} Consider the SDE
\begin{equation}\label{N97}
\aligned dx(t)=\begin{pmatrix}
-x_1(t) - 2x_{1}^3(t) - x_{2}(t)\\
-x_{2}(t)+x_1(t)-2x_2^3(t)\\
\end{pmatrix}dt+|x|xdB_{t}\endaligned
\end{equation}
with the initial value $x_{0} = (1,-1)^{T}$.
First, we have
\begin{align*}
x^{T}f(x)+\frac{7-1}{2}|g(x)|^{2}=-x_1^2-x_2^2-0.5(x_1^4 + x_2^4) \le -|x|^{2}.
\end{align*}
So $p=7$ and $\lambda=1$ for Assumption \ref{E21}, By Theorem 4.4 in \cite{AA}, we know the 6th moment of the unique global solution of (\ref{N97}) is exponentially stable with order 7.

We can get easily $K_{R}=18R^{2}$. If we choose $h(\Delta)=\Delta^{-\frac{1}{13}}$, then we have $K^{6}_{h(\Delta)}\Delta=18^{6}\Delta^{\frac{1}{13}}\rightarrow 0$. Then by Theorem \ref{E3}, we have that for any $\varepsilon<1$, we can choose $\Delta$ sufficiently small such that
\begin{align*}
\lim_{k\rightarrow \infty}\sup\frac{\log\mathbb{E}(|Y_{k}|^{7})}{k\Delta}\le -7(1-\varepsilon)
\end{align*}
and
\begin{align*}
\lim_{t\rightarrow \infty}\sup\frac{\log\mathbb{E}(|Y(t)|^{7})}{t}\le -7(1-\varepsilon).
\end{align*}
Thus, the modified truncated Milstein method is close to the exponential stability of the exact solution for the given SDE.

We choose $\Delta=2^{-10}$ and $k=5\cdot2^{12}$ for computer simulation, then we obtain Figure 3 and 4. Figure 3 clearly suggests that $Y^{(1)}_{k}$ is asymptotically stable.

\begin{figure}[htbp]
 \centering
  \includegraphics[width=0.3\textwidth]{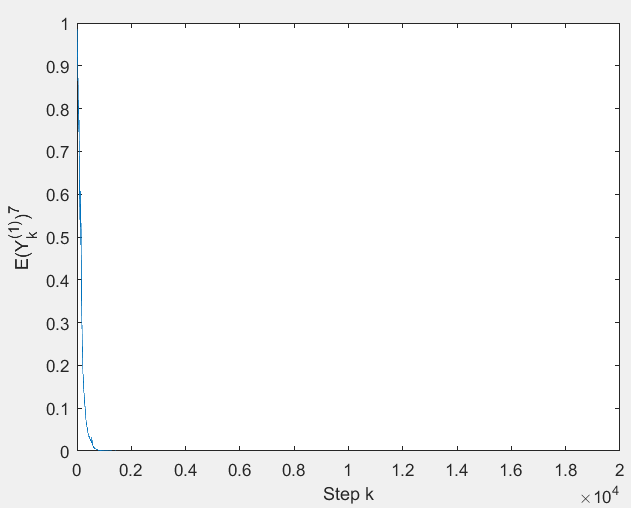}
  \caption{Path of $E(Y_k^{(1)})^7$.}\label{fig:digit}
\end{figure}

Figure 4 implies that $\lim_{k\rightarrow \infty}\sup\frac{\log\mathbb{E}(|Y_{k}^{(1)}|^{7})}{k\Delta}\le-7<0$ for $k$ large enough.
\begin{figure}[htbp]
 \centering
  \includegraphics[width=0.3\textwidth]{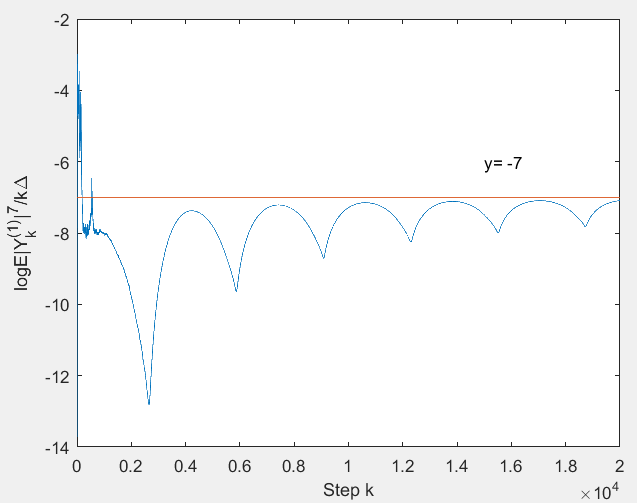}
  \caption{Path of $\frac{\log E|Y_k^{(1)}|^7}{k\Delta}$.}\label{fig:digit}
\end{figure}

For $Y^{(2)}_{k}$,  we obtain Figure 5 and 6, which imply that $Y^{(2)}_{k}$ is asymptotically exponentially stable.

\begin{figure}[htbp]
 \centering
  \includegraphics[width=0.3\textwidth]{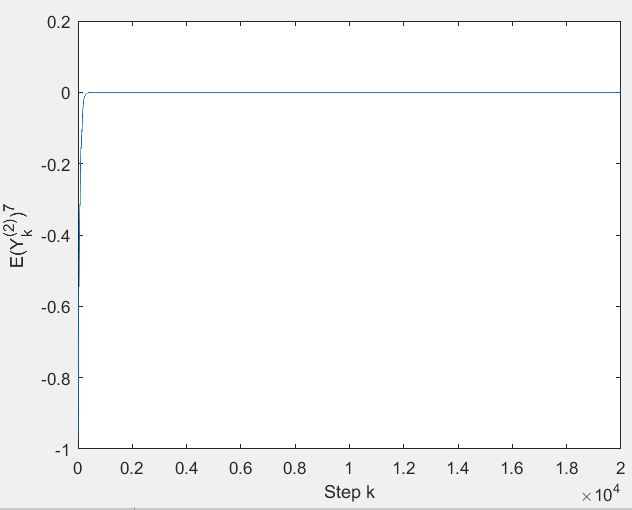}
  \caption{Path of $E(Y_k^{(2)})^7$}.\label{fig:digit}
\end{figure}

\begin{figure}[htbp]
 \centering
  \includegraphics[width=0.3\textwidth]{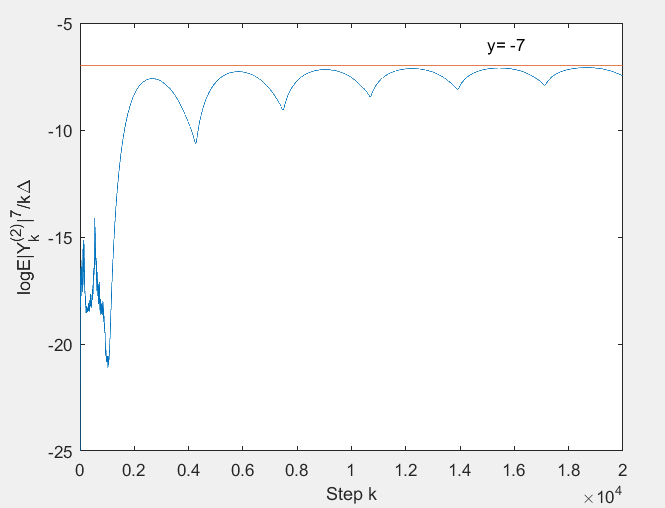}
  \caption{Path of $\frac{\log E|Y_k^{(2)}|^7}{k\Delta}.$}\label{fig:digit}
\end{figure}

\end{document}